\documentclass{article}
\usepackage{amsthm}
\usepackage{amsmath}
\usepackage{amssymb}
\usepackage{amscd}
\usepackage{graphicx}

\newtheorem{theorem}[subsection]{Theorem}
\newtheorem{lemma}[subsection]{Lemma}

\newtheorem{corollary}[subsection]{Corollary}

\theoremstyle{definition}

\newtheorem{definition}[subsection]{Definition}

\theoremstyle{remark}

\newtheorem{remark}[subsection]{Remark}

\def\R{\mathbb{R}}

\vspace{1.5cm}
\title{{\bf Newton flows for elliptic functions IV}\\
{\small {\bf Pseudo Newton graphs: bifurcation \& creation of flows}}}

\author{G.F. Helminck,\\
Korteweg-de Vries Institute\\
University of Amsterdam\\
P.O. Box 94248\\
1090 GE Amsterdam\\
The Netherlands\\
e-mail: g.f.helminck@uva.nl\\ 
F. Twilt,\\
Department of Applied Mathematics\\
University of Twente\\
P.O. Box 217, 7500 AE Enschede\\
The Netherlands\\
e-mail: f.twilt@kpnmail.nl\\
}
 
\begin{document}
\maketitle

\begin{abstract}
\noindent
An elliptic Newton flow is a dynamical system that can be interpreted as a continuous version of Newton's iteration method for finding the zeros of an elliptic function $f$. Previous work focusses on structurally stable  flows (i.e., the phase portraits are topologically invariant under perturbations of the poles and zeros for $f$), including a classification / representation result for such flows in terms of {\it Newton graphs}  (i.e., cellularly embedded toroidal graphs fulfilling certain combinatorial properties).
The present paper deals with {\it non-structurally stable elliptic Newton flows} determined by pseudo Newton graphs (i.e., cellularly embedded toroidal graphs, either generated  by a Newton graph, or the so called nuclear Newton graph, exhibiting only one vertex and two edges).  Our study results into a deeper insight in the creation of structurally stable Newton flows and the bifurcation of non-structurally stable Newton flows.
\end{abstract}

{\bf Subject classification:} 
05C75, 
33E05, 34D30, 
37C70, 49M15.\\

{\bf Keywords:}  
Dynamical system, desingularized elliptic Newton flow, structural stability, elliptic function, phase portrait, Newton graph (elliptic-, nuclear-, pseudo-), cellularly embedded toroidal (distinguished) graph, face traversal procedure,
Angle property, Euler property, Hall condition.

\section{Motivation; recapitulation of earlier results}

In order to clarify the context of the present paper, we recapitulate some earlier results.\\

\noindent
{\large{\bf {\small 1.1 Elliptic Newton flows; structural stability}}}\\

The results in the following four subsections, can all be found in our paper  \cite{HT1}.\\

\noindent
{\large{\bf {\small 1.1.1 Planar and toroidal elliptic Newton flows}}}\\

Let $f$  be an elliptic (i.e., meromorphic, doubly periodic) function of order $r(\geqslant 2)$ on the complex plane $\mathbb{C}$ with $(\omega_{1},\; \omega_{2})$, ${\rm Im}\frac{\omega_{2}}{\omega_{1}} >0,$ as {\it basic periods} spanning a lattice $\Lambda(=\Lambda_{\omega_{1},\; \omega_{2}})$.

The {\it planar elliptic Newton flow} $\overline{\mathcal{N} }(f)$ is a 
 $C^{1}$-vector field on $\mathbb{C}$, defined as a {\it desingularized version}\footnote{\label{FTN1}
 In fact, we consider the system $  \dfrac{dz}{dt} =-(1+|f(z)|^{4})^{-1}|f'(z)|^{2}\dfrac{f (z)}{f^{'} (z)}$: a continuous version of Newton's damped iteration method for finding zeros for $f$.
} of the planar dynamical system, $\mathcal{N}(f)$, given by: 
\begin{equation}
\label{vgl2x}
  \dfrac{dz}{dt} = \dfrac{-f (z)}{f^{'} (z)}, z \in \mathbb{C}.
\end{equation}

\noindent
On a  non-singular, oriented $\overline{\mathcal{N} }(f)$-trajectory $z(t)$ we have: \\

\noindent
- arg $(f)=$constant, and $|f(z(t))|$ is a strictly decreasing function on $t$.\\
So that the $\overline{\mathcal{N} }(f)$-{\it equilibria} are of the form:
\ \\
\noindent
- a stable star node (attractor); in the case of a zero for $f$, or \\
- an unstable star node (repellor); in the case of a pole for $f$, or\\
- a saddle; in the case of a critical point for $f$ (i.e., $f'$ vanishes, but $f$ not).\\

For an (un)stable node the (outgoing) incoming trajectories intersect under a non-vanishing angle $\frac{\Delta}{k}$, where $\Delta$ stands for the difference of the arg$f$-values on these trajectories, and $k$ for the multiplicity of the corresponding (pole) zero. 
The saddle in the case of a {\it simple} critical point (i.e., $f '$ does not vanish) is {\it orthogonal} and the two unstable (stable) separatrices constitute the ``local'' {\it unstable (stable) manifold} at this saddle.\\

Functions such as $f$ correspond to meromorphic functions on the torus $T(\Lambda)(=\mathbb{C}/\Lambda_{\omega_{1},\; \omega_{2}})$.
So, we can interprete $\overline{\mathcal{N} }(f)$ as a global
 $C^{1}$-vector field, denoted\footnote{\label{FTN2} Occasionally, we will refer to $\overline{\overline{\mathcal{N}}} (f)$ as to a toroidal (elliptic) Newton flow.} $\overline{\overline{\mathcal{N}}} (f)$, on the Riemann surface $T(\Lambda)$ and it is allowed to apply results for  $C^{1}$-vector fields on compact differential manifolds, such as certain theorems of Poincar\'e-Bendixon-Schwartz (on limiting sets) and those of Baggis-Peixoto (on  $C^{1}$-structural stability). \\
 It is well-known that the function $f$ has precisely $r$ zeros and $r$ poles  (counted by multiplicity) on the half open / half closed period parallelogram $P_{\omega_{1}, \; \omega_{2}}$ given by $$\{t_{1}\omega_{1} + t_{2}\omega_{2} \mid 0 \leqslant t_{1}<1, \;0 \leqslant t_{2}<1\}.$$

\noindent
Denoting
these zeros and  poles 
by $a_{1}, \! \cdots \!\!,a_{r} $, resp. $b_{1}, \! \cdots \! ,b_{r} $,  
we have\footnote{\label{FTN3}In fact $\lambda_{0}= -\eta(f(\gamma_{2}))\omega_{1}+ \eta(f(\gamma_{1}))\omega_{2},$ where $\eta(f(\cdot))$ stands for winding numbers of the curves $f(\gamma_{1})$ and $f(\gamma_{2})$ with $\gamma_{i}$ the sides of $P_{\omega_{1}, \; \omega_{2}}$  spanned by $\omega_{i} , i=1, 2,$ (cf. \cite{Peix1}).}: (cf. \cite{M2})
\begin{equation}
\label{vgl9}
a_{i} \neq b_{j},i, j=1, \! \cdots \! \!,r \text{ and }a_{1}+  \cdots +a_{r} =b_{1} + \cdots  +b_{r} +\lambda_{0}, \lambda_{0} \in \Lambda,
\end{equation}
and thus 
\begin{equation}
\label{vgl10}
[a_{i}] \neq [b_{j}],i, j=1, \! \cdots \! ,r \text{ and }[a_{1}]+ \! \cdots \! +[a_{r}] =[b_{1}] +\! \cdots \! +[b_{r}].
\end{equation}
where $[a_{1}], \! \cdots \!, [a_{r}] $ and  $[b_{1}], \! \cdots \!, [b_{r}] $ are the zeros, resp. poles for $f$ on $T(\Lambda)$ and $[\cdot]$ stands for the congruency class $\mathrm {mod} \, \Lambda$
of a number in $\mathbb{C}$. Conversely, any pair $(\{a_{1}, \! \cdots \!\!,a_{r} \}, \{b_{1}, \! \cdots \! ,b_{r}\}) $
that fulfils (\ref{vgl9}) determines (up to a multiplicative constant) an elliptic function with $\{a_{1}, \! \cdots \!\!,a_{r} \}$ and $\{b_{1}, \! \cdots \! ,b_{r}\}$ as zeros resp. poles in $P(=P_{\omega_{1}, \; \omega_{2}})$.\\

\noindent
{\large{\bf {\small 1.1.2 The topology $\tau_{0}$}}}\\

It is not difficult to see that the functions $f$, and also the corresponding toroidal Newton flows, can be represented by the set of all ordered pairs 
$
 (\{ [a_{1}],\! \cdots \! ,[a_{r} ]\}, \, \{ [b_{1}], \! \cdots \! ,[b_{r} ]\})
$
of congruency classes $\mathrm {mod} \, \Lambda$ (with $a_{i}, b_{i} \in P, i=1, \dots, r,$) that fulfil 
(\ref{vgl10}).

This representation space can be endowed with a topology, say $\tau_{0}$, induced by the Euclidean topology on $\mathbb{C}$, that is natural in the following sense: 
Given an elliptic function $f$ of order $r$ and $\varepsilon >0$ sufficiently small, a $\tau_{0}$-neighborhood $\mathcal{O}$ of $f$ exists such that for any $g$ in $\mathcal{O}$, the zeros (poles) for $g$ are contained in $\varepsilon$-neighborhoods of the zeros (poles) for $f$.

$E_{r}(\Lambda)$  is the set of all functions $f$ of order $r$ on $T(\Lambda)$ and $N_{r}(\Lambda)$ the set of corresponding flows  $\overline{\overline{\mathcal{N}} }(f)$. By $X(T)$ we mean the set of all $C^{1}$-vector fields on $T$, endowed with the $C^{1}$-topology (cf. \cite{Hir}). 

The topology $\tau_{0}$ on $E_{r}(\Lambda$
and the $C^{1}$-topology  on $X(T)$ are matched by: \\

\noindent
The map $E_{r}(\Lambda) \rightarrow X(T): f \mapsto \overline{\overline{\mathcal{N}}} (f)$
is $\tau_{0}\!-\!C^{1}$ continuous.\\

\noindent
{\large{\bf {\small 1.1.3 Canonical forms of elliptic Newton flows}}}\\

\noindent
The flows $\overline{\overline{\mathcal{N}} }(f)$ and $\overline{\overline{\mathcal{N}} }(g)$) in $N_{r}(\Lambda)$ are called {\it conjugate}, denoted $\overline{\overline{\mathcal{N}}} (f) \sim \overline{\overline{\mathcal{N}}} (g)$, if there is a homeomorphism from $T$ onto itself mapping maximal trajectories of $\overline{\overline{\mathcal{N}} }(f)$ onto those of $\overline{\overline{\mathcal{N}} }(g)$, thereby respecting the orientations of these trajectories.\\

\noindent
For a given $f$ in $E_{r}(\Lambda)$, let the lattice $\Lambda^{*}$  be arbitrary. Then there is a function $f^{*}$ in $E_{r}(\Lambda^{*})$ such that:
$$
\overline{\overline{\mathcal{N}}} (f) \sim \overline{\overline{\mathcal{N}}} (f^{*}).
$$
In fact, the linear isomorphism $\mathbb{C} \to \mathbb{C}: (\omega_{1},\omega_{2}) \mapsto (\omega_{1}^{*},\omega_{2}^{*})$; ${\rm Im}\frac{\omega_{2}^{*}}{\omega_{1}^{*}} >0,$ transforms $\Lambda$ into $\Lambda^{*}$ and the pair $(\{a_{1}, \! \cdots \!\!,a_{r} \}, \{b_{1}, \! \cdots \! ,b_{r}\}) $, determining $f$, into $(\{a_{1}^{*}, \! \cdots \!\!,a_{r}^{*} \}, \{b_{1}^{*}, \! \cdots \! ,b_{r}^{*}\}) $ fixing $f^{*}$ in $E_{r}(\Lambda^{*})$;
if $\Lambda=\Lambda^{*}$, then this linear isomorphism is unimodular.\\
\noindent
It is always possible to choose\footnote{\label{FTN4} $\tau$ satisfies $\text{Im } \tau  > 0, | \tau | \geqslant 1, - \frac{1}{2} \leqslant  \text{Re }  \tau  <\frac{1}{2},  \text{Re } \tau \leqslant 0, \text{ if }    | \tau |=1.$} $\Lambda^{*}(=\Lambda_{1,\tau})$, where $(1, \tau)$ is a {\it reduced}\footnote{\label{FTN5} A pair $(\omega_{1}^{*},\omega_{2}^{*})$ of basic periods for $f^{*}$  is called {\it reduced} if $|\omega_{1}^{*}|$ is minimal among all periods for $f^{*}$ , whereas $| \omega_{2}^{*} |$ is minimal among all periods $\omega$ for $f^{*}$  with the property ${\rm Im}\frac{\omega}{\omega_{1}^{*}} >0$.}
pair of periods for $f^{*}$ and to subsequently apply the linear transformation $(1,\tau) \mapsto (1,i)$
, so that we even may assume that $(1,i)$ is a pair of reduced periods for the corresponding elliptic function on $\Lambda_{1,i}$.\\

\noindent
Consequently, unless strictly necessary, we suppress the role of $\Lambda$ and write: $E_{r}(\Lambda)=E_{r}$, $T(\Lambda)=T$ and $N_{r}(\Lambda)=N_{r}$.\\

\noindent
{\large{\bf {\small 1.1.4 Structural stability}}}\\

\noindent
The flow $\overline{\overline{\mathcal{N}}} (f) $ 
in $N_{r}$ is called 
{\it $\tau_{0}$-structurally stable}, if there is a $\tau_{0}$-neighborhood $\mathcal{O}$ of $f$, such that for all $g \in \mathcal{O}
$ we have: $\overline{\overline{\mathcal{N}}} (f) \sim \overline{\overline{\mathcal{N}}} (g)$; the set of all $\tau_{0}$-structurally stable flows $\overline{\overline{\mathcal{N}}} (f)$
 is denoted by $\tilde{N}_{r}$. \\
 
 \noindent
$C^{1}$-structural stability for $\overline{\overline{\mathcal{N}}} (f)$ implies $\tau_{0}$-structurally stability for $\overline{\overline{\mathcal{N}}} (f)$; see 
Subsubsection 1.1.2.
So, when discussing structurally stable toroidal Newton flows we skip the adjectives $\tau_{0}$ and $C^{1}$.\\

\noindent
- A structurally stable $\overline{\overline{\mathcal{N}}} (f)$ has precisely 2$r$ different simple saddles (all orthogonal).\\

\noindent
Note that if $\overline{\overline{\mathcal{N}}} (f)$ is structurally stable, then also $\overline{\overline{\mathcal{N}}} (\frac{1}{f})$, because 
$\overline{\overline{\mathcal{N}}} (\frac{1}{f})=-\overline{\overline{\mathcal{N}}} (f)$ [{\bf duality}].\\

\noindent
The main results obtained in \cite{HT1} are:\\

\noindent
- $\overline{\overline{\mathcal{N}}} (f) \in \tilde{N}_{r}$ if and only if the function $f$  is non-degenerate\footnote{\label{FTN6} i.e., all {\it zeros}, {\it poles} and {\it critical  points} for $f$ are simple, and no critical points are connected by $\overline{\overline{\mathcal{N}}} (f)$-trajectories.
} [{\bf characterization}].\\
\noindent
- The set of all non-degenerate functions of order $r$ is open and dense in $E_{r}$ [{\bf genericity}].\\

\noindent
{\large{\bf {\small 1.2 Classification \& representation of structurally stable elliptic Newton flows}}}\\

\noindent
The following three subsubsections describe shortly the main results from our paper  \cite{HT2}.\\

\noindent
{\large{\bf {\small 1.2.1 The graphs $\mathcal{G}(f )$ and $\mathcal{G}^{*}(f)$}}}\\

\noindent
For the flow $\overline{\overline{\mathcal{N}}} (f)$
in $\tilde{N}_{r},$ we define the connected (multi)graph\footnote{\label{FTN7}The graph $\mathcal{G}(f )$ has no loops, basically because the zeros for $f$ are simple.} $\mathcal{G}(f )$ of order $r$ on $T$ by:\\

\noindent
-vertices  are the $r$ zeros for $f$;\\
-edges are the 2$r$ unstable manifolds at the critical points for $f$;\\
-faces are the $r$ basins of repulsion of the poles for $f$.\\

\noindent
Similarly, we define the toroidal graph $\mathcal{G}^{*}(f)$ at the {\it repellors, stable manifolds} and {\it basins of attraction} for $\overline{\overline{\mathcal{N}}} (f)$. Apparently,  $\mathcal{G}^{*}(f)$ is the {\it geometrical dual} of $\mathcal{G}(f)$; see Fig. 1. \\

\noindent
The following features of $\mathcal{G}(f)$ reflect the main properties of the phase portraits of $\overline{\overline{\mathcal{N}}} (f)$: \\

\noindent
- $\mathcal{G}(f )$ is {\bf cellularly embedded}, i.e., each  face is homeomorphic to an open $\mathbb{R}^2$-disk;\\
- all (anti-clockwise measured) angles at an attractor in the boundary of a face that span a
sector of this face, are strictly positive and sum up to $2\pi$; [{\bf A-property}] \\
- the boundary of each $\mathcal{G}(f )$-face - as subgraph of $\mathcal{G}(f )$ - is {\it Eulerian}, i.e., admits a closed facial
walk that traverses each edge only once and goes through all vertices [{\bf E-property}].\\

\begin{figure}[h!]
\begin{center}
\includegraphics[scale=0.45]{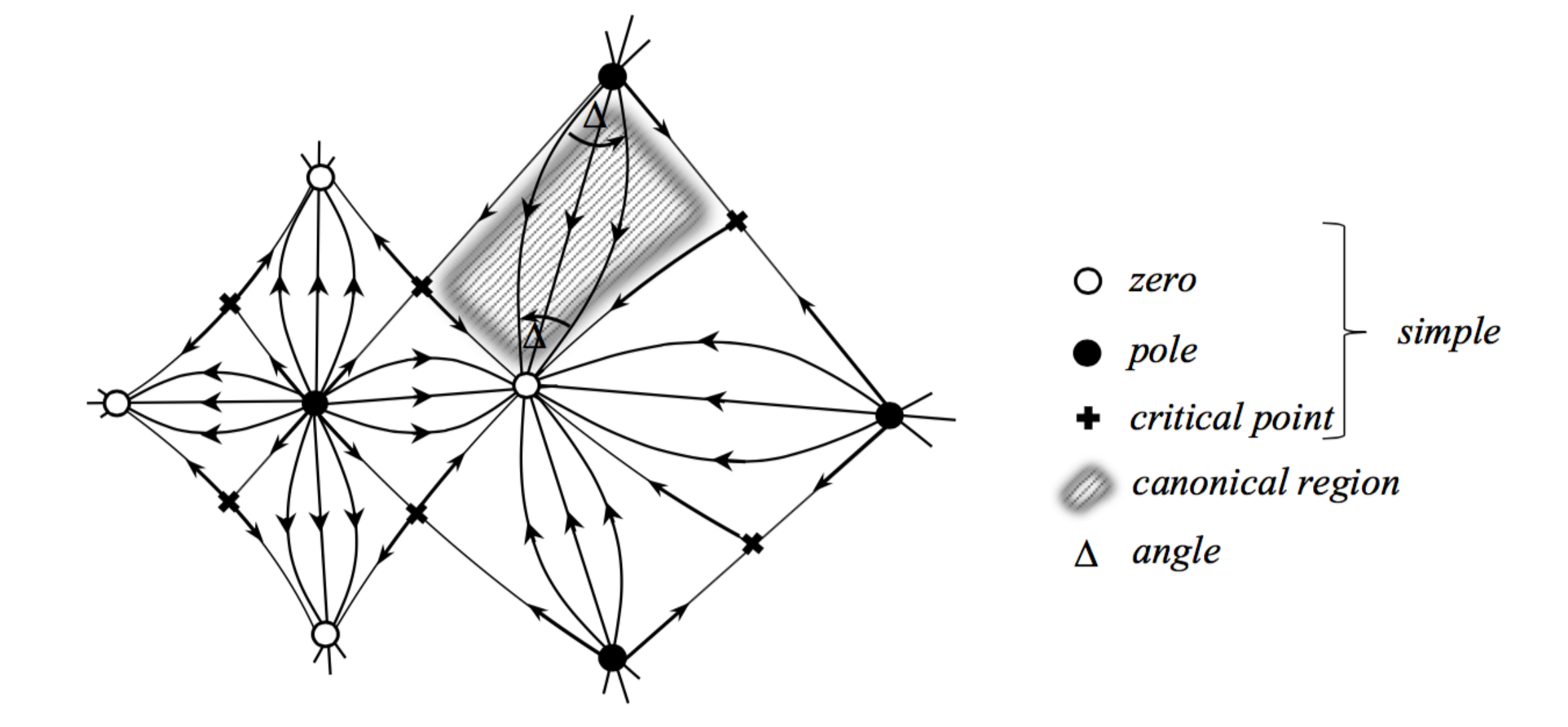}
\caption{\label{Figure2N} Basin of repulsion (attraction) of $\overline{\overline{\mathcal{N}}} (f)$ for a pole (zero) of $f$.}
\end{center}
\end{figure}

\noindent
The anti-clockwise permutation on the embedded edges at vertices of $\mathcal{G}(f )$ induces a clockwise
orientation of the facial walks on the boundaries of the $\mathcal{G}(f )$-faces; see Fig.\ref{Figure13}. On its turn, the
clockwise orientation of $\mathcal{G}(f )$-faces gives rise to a clockwise permutation on the embedded edges at
$\mathcal{G}^*(f )$-vertices, and thus to an anti-clockwise orientation of $\mathcal{G}^*(f )$. In the sequel, all graphs of the type
$\mathcal{G}(f )$, $\mathcal{G}^*(f )$, are always oriented in this way, see Fig.\ref{Figure13}. Hence, by duality,
\begin{equation}
\label{lb14}
\mathcal{G}^*(f )=-\mathcal{G}(\frac{1}{f}).
\end{equation}
It follows that also $\mathcal{G}^*(f )$ is {\it cellularly embedded} and fufills the {\it E-} and {\it A-property}.\\

\noindent
{\large{\bf {\small 1.2.2 Newton graphs}}}\\

\noindent
A connected multigraph $\mathcal{G}$ in $T$ with $r$ vertices, $2r$ edges and $r$ faces is called {\it a Newton graph} (of order $r$) if this
graph is {\it cellularly embedded} and moreover, the {\it A-property}
and the {\it E-property} hold.

\noindent
It is proved that the dual $\mathcal{G}^*$ of a Newton graph $\mathcal{G}$ is also {\it Newtonian} (of order $r$).
The anti-clockwise (clockwise) permutations on the edges of $\mathcal{G}$ at its vertices endow a clockwise
(anti-clockwise) orientation of the $\mathcal{G}$-faces and, successively, an anti-clockwise (clockwise)
orientation of the $\mathcal{G}^*$-faces, compare Fig.\ref{Figure13}).\\

\noindent
-Apparently $\mathcal{G}(f )$ and $\mathcal{G}^*(f )$ are Newton graphs.\\

\noindent
The main results obtained in \cite{HT2} are:

\noindent
-If $\overline{\overline{\mathcal{N}}} (f)$ and $\overline{\overline{\mathcal{N}}} (g)$ are structurally stable and of the same order, then:
$$
\overline{\overline{\mathcal{N}}} (f) \sim \overline{\overline{\mathcal{N}}} (g) \Leftrightarrow \mathcal{G}(f ) \sim \mathcal{G}(g ) \text{    {\bf [classification]}}.
$$
\noindent
-given a clockwise oriented Newton graph $\mathcal{G}$ of order $r$, there exists a structurally stable Newton flow $\overline{\overline{\mathcal{N}}} (f_{\mathcal{G}})$ such that $\mathcal{G}(f_{\mathcal{G}}) \sim \mathcal{G}$ and thus: ($\mathcal{H}$ is another clockwise oriented Newton graph) 
$$
\overline{\overline{\mathcal{N}}} (f_{\mathcal{G}}) \sim \overline{\overline{\mathcal{N}}} (f_{\mathcal{H}}) \Leftrightarrow \mathcal{G} \sim \mathcal{H} \text{    {\bf [representation]}}.
$$
Here the symbol $\sim$ between flows stands for conjugacy, and between graphs for equivalency (i.e. an orientation preserving isomorphism).\\

\begin{figure}[h]
\begin{center}
\includegraphics[scale=1.0]{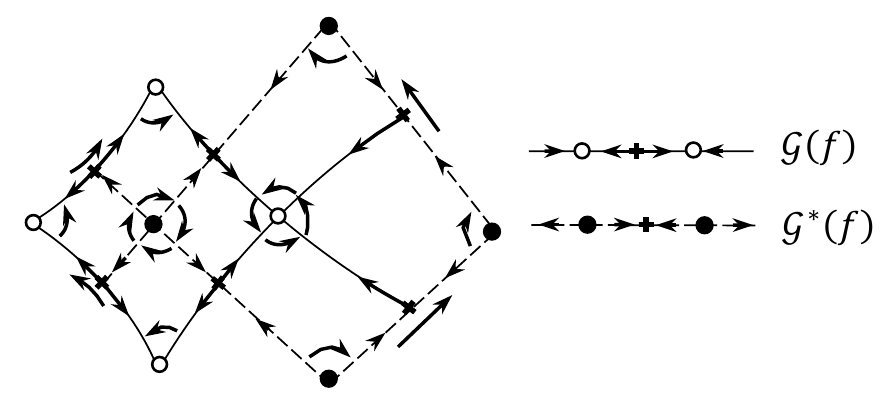}
\caption{\label{Figure13} The oriented graphs $\mathcal{G}(f )$ and $\mathcal{G}^{*}(f )$.}
\end{center}
\end{figure}

\noindent
{\large{\bf {\small 1.2.3 Characteristics for the A-property and the E-property}}}\\

Let $\mathcal{G}$ be a cellularly embedded graph of order $r$ in $T$, not necessarily fulfilling the {\it A-property}.
There is a simple criterion available for $\mathcal{G}$ to fulfil the {\it A-property}. In order to formulate this
criterion, we denote the vertices and faces of $\mathcal{G}$ by $v_i$ and $F_j$ respectively, $i, j=1, \cdots ,r$.
Now, let $J$ be a subset of $\{ 1, \cdots ,r\}$ and denote the subgraph of $\mathcal{G}$, generated by all vertices and edges incident with the faces $F_j, j \in J$, by $\mathcal{G}(J)$.The set of all vertices in $\mathcal{G}(J)$ is denoted by $V(\mathcal{G}(J))$. Then: 
\begin{equation*}
\text{$\mathcal{G}$ has {\it the A-property}} \Leftrightarrow
\text{ $|J  |<|V(\mathcal{G}(J))|$ for all $J$, $\emptyset \neq \!J\! \subsetneq \{1, \! \cdots \! ,r \}$,   [{\bf Hall condition}]}
\end{equation*}
where $|\cdot |$ stands as usual for cardinality. 
As a by-product we have:\\

\noindent
-Under the {\it A-property}, the set of exterior $\mathcal{G}(J)$-vertices (i.e., vertices in $\mathcal{G}(J)$ that are also adjacent to $\mathcal{G}$-faces, but not in $\mathcal{G}(J)$), is {\it non-empty}.\\
-Let $\mathcal{G}$ be a cellularly embedded graph of order $r$ in $T$, not necessarily fulfilling the {\it E-property}.  
We consider the {\it rotation system} $\Pi$ for $\mathcal{G}$:
$$
\Pi=\{ \pi_{v} \!\mid \! \text{all vertices $v$ in $\mathcal{G}$} \},
$$
where the {\it local  rotation system} $\pi_{v}$ at $v$ is the cyclic permutation of the edges incident with $v$ such that $\pi_{v}(e)$ is the successor of $e$ in the anti-clockwise ordering around $v$. Then, the boundaries of the faces of $\mathcal{G}$ are formally described by $\Pi$-walks as: \\

\noindent
If $e(=(v'v''))$ stands for an edge, with end vertices $v'$ and $v''$, we define a $\Pi$-{\it walk} ({\it facial walk}), say
$w$, on $\mathcal{G}$ as follows: [{\bf face traversal procedure}]\\

\noindent
Consider an edge $e_{1}=(v_{1}v_{2})$ and the closed walk
$w=v_{1}e_{1}v_{2}e_{2}v_{3} \cdots v_{k}e_{k}v_{1}$, which is determined by the requirement that, for $i=1, \cdots , \ell,$ we have $\pi_{v_{i+1}}(e_{i})=e_{i+1}$, where $e_{\ell +1}=e_{1}$ and $\ell$ is minimal.\footnote{\label{FTN3}Apparently, such "minimal" $l$ exists since $\mathcal{G}$ is finite. In fact, the first edge which is repeated in the same direction when 
traversing $w$, is $e_1$ (cf. \cite{MoTh}).}
\\ 

\noindent
Each edge occurs either once in two different $\Pi$-walks, or twice (with opposite orientations) in only
one $\Pi$-walk. $\mathcal{G}$ has the {\it E-property} iff the first possibility holds for all $\Pi$-walks. The dual $\mathcal{G}^{*}$ admits a
loop\footnote{\label{vtn9}Note that by assumption $\mathcal{G}$ has no loops.} iff the second possibility occurs at least in one of the $\Pi$-walks.

\noindent
The following observation will be referred to in the sequel:\\

\noindent
Under the {\it E-property} for $\mathcal{G}$, each $\mathcal{G}$-edge is adjacent to different faces; in fact,
any $\mathcal{G}$-edge, say $e$, determines precisely one $\mathcal{G}^{*}$-edge $e^{*}$ (and vice versa) so that
there are $2r$ ÒintersectionsÓ $s=(e, e^{*} )$ of $\mathcal{G}$- and $\mathcal{G}^{*}$-edges. We consider an abstract graph with these pairs $s$, together with the $\mathcal{G}$- and $\mathcal{G}^{*}$-vertices, as vertices, two of the vertices of this abstract graph being connected {\it iff} they are incident with the same $\mathcal{G}$- edge or $\mathcal{G}^{*}$-edge. This graph admits a cellular embedding in $T$ (cf. \cite{HT2}), which will be referred to as to the ``distinguished'' graph $\mathcal{G} \wedge \mathcal{G}^{*}$ with as faces the so-called {\it canonical regions} (compare Fig.3). Following \cite{Peix1}, $\mathcal{G} \wedge \mathcal{G}^{*}$ determines a $C^{1}$-structurally stable flow $X(\mathcal{G})$ on $T$. In \cite{HT2} we proved that, if the {\it $A$-property} holds as well, $X(\mathcal{G})$ is topologically equivalent with a structurally stable elliptic Newton flow of order $r$.\\

\noindent
{\large{\bf {\small 1.3 The Newton graphs of order $r, r=2,3$ }}}\\

\noindent
Following \cite{HT3}, we present the lists of all -up to duality and conjugacy - Newton graphs of order $r, r=2,3$. 

\begin{figure}[h!]
\begin{center}
\includegraphics[scale=0.2]{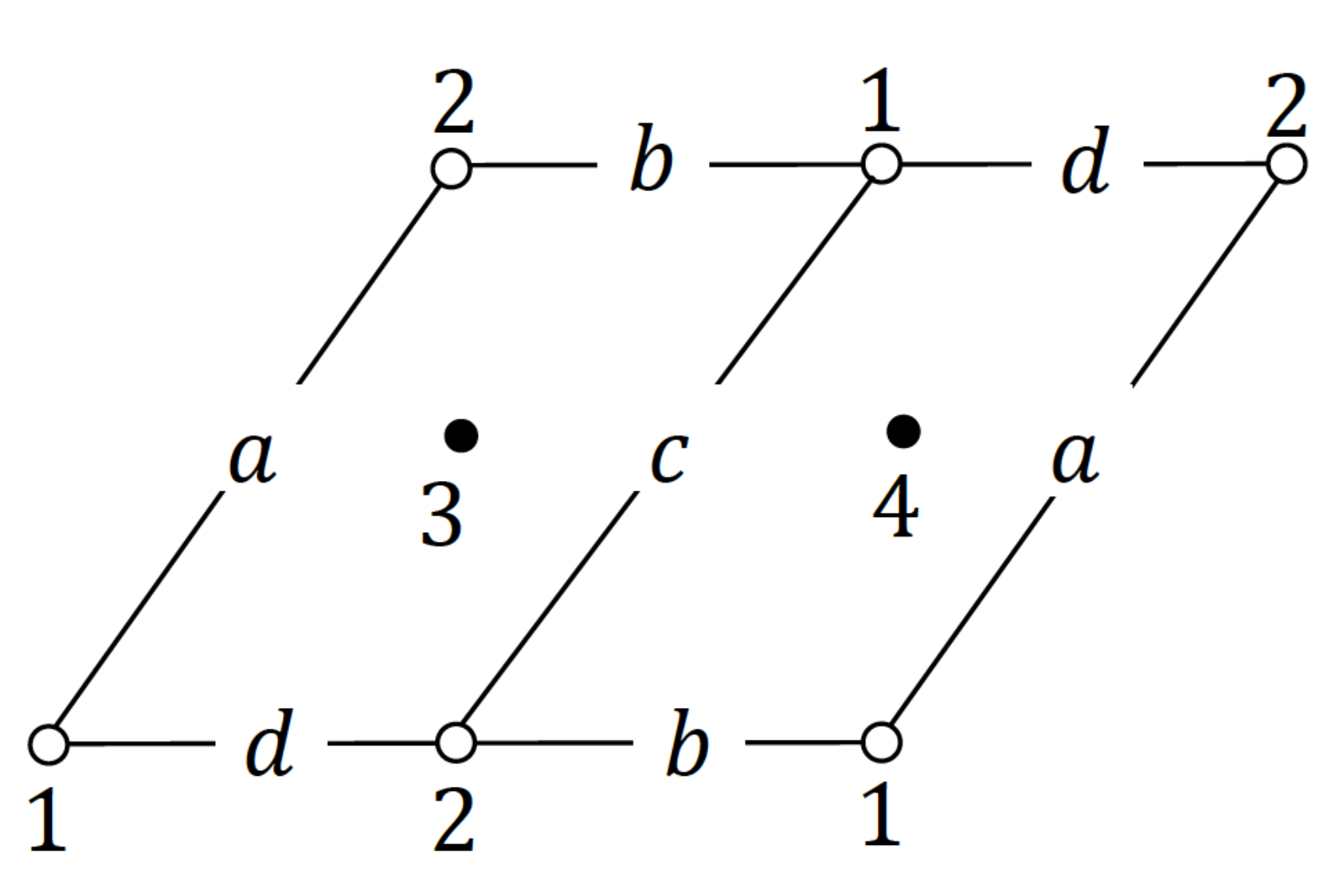}
\caption{\label{Figure9.12} The  $2^{nd}$ order Newton graph.}
\end{center}
\end{figure}

\begin{figure}[h!]
\begin{center}
\includegraphics[scale=0.6]{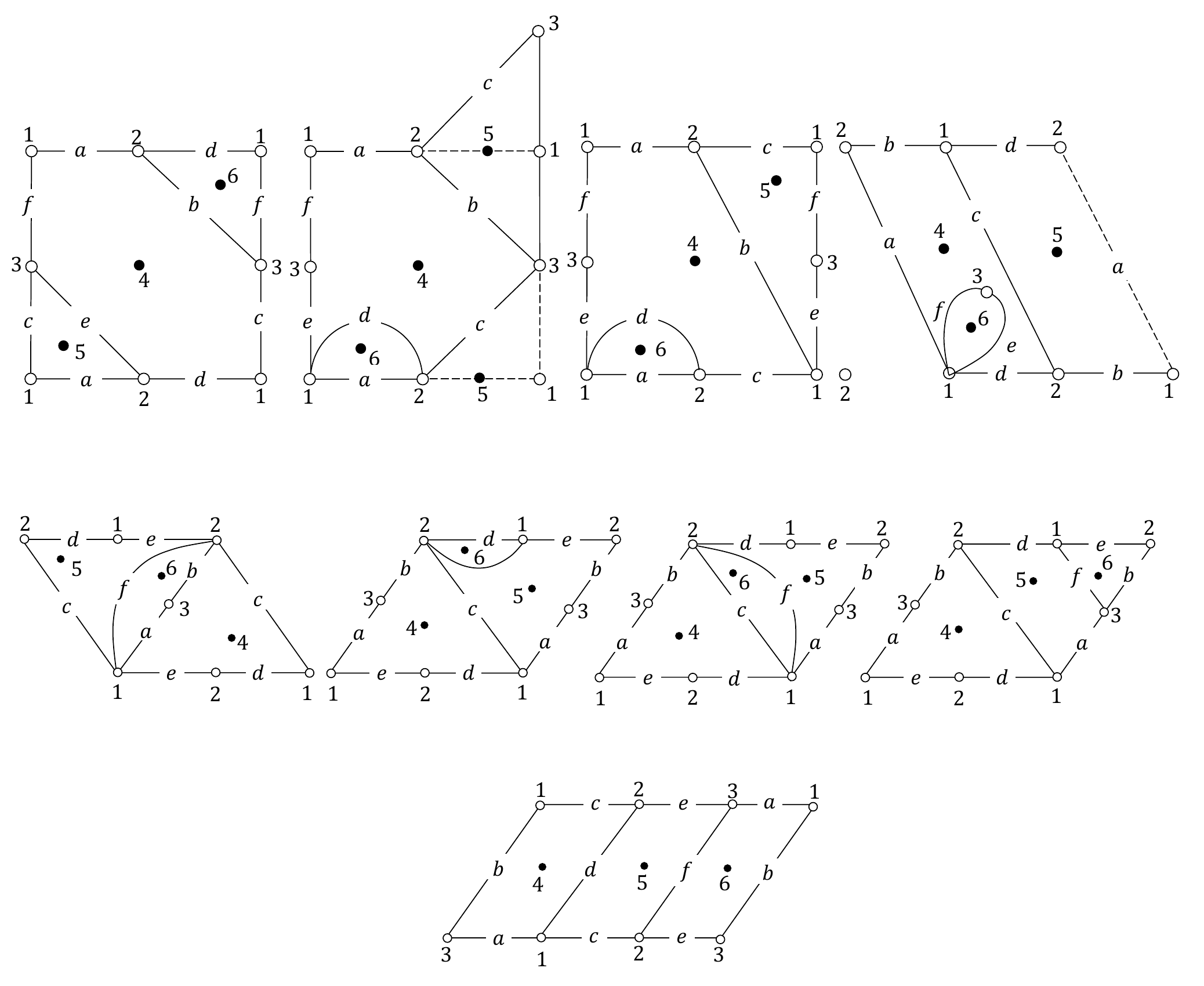}
\caption{\label{Figure9.11} The  $3^{rd}$ order Newton graphs.}
\end{center}
\end{figure}

\section{Pseudo Newton graphs }

Throughout this section, let 
$\mathcal{G}_{r}$ 
be a Newton graph
of order $r$.

Due to the {\it E-property}, we know that an arbitrary edge of $\mathcal{G}_{r}$ is contained in precisely two different faces. If we delete such an edge from $\mathcal{G}_{r}$ and merge the involved faces $F_{1}, F_{2}$ into a new face, say $F_{1,2}$, we obtain a  toroidal connected multigraph (again cellularly embedded) with $r$ vertices, $2r-1$ edges and $r-1$ faces: $F_{1,2}$, $F_{3}, \cdots , F_{r}$.\\
If $r=2$, then this graph has only one face.\\
If $r>2$, put $J=\{1,2\}$, thus $\emptyset \neq J \subsetneq \{1, \! \cdots \! ,r \}$. Then, we know, by the A-property 
(cf. Subsubsection 1.2.3)
, that the set Ext$(\mathcal{G}(J))$ of exterior $\mathcal{G}(J)$-vertices is non-empty. 
Let $v \in $ Ext$(\mathcal{G}(J))$, thus $v \in \partial F_{1,2} $.
Hence, $v$ is incident with an edge, adjacent only to  one of the faces $F_{1}$, $F_{2}$.
Delete this edge and obtain the ``merged face'' $F_{1,2,3}$. \\
If $r=3$, the result is a graph with only {\it one} face.

If $r >3$, put $J=\{1,2,3 \}$. By the same reasoning as used in the case $r=3$, it can be shown 
that $\partial F_{1,2,3}$ contains an edge belonging to another face than $F_{1}, F_{2} \text{ or }F_{3}$, say $F_{4}$ . 
Delete this edge and obtain the ``merged face'' $F_{1,2,3,4}$.
And so on. In this way, we obtain - in $r\!-\!1$ steps - a connected cellularly embedded multigraph, say $\check{\mathcal{G}}_{r}$
, with $r$ vertices, $r+1$ edges and only one face. 

Obviously, $\check{\mathcal{G}}_{r}$ contains vertices of degree $\geqslant 2$. Let us assume that there exists a vertex for $\check{\mathcal{G}}_{r}$, say $v$, with deg($v$)=1. If we delete this vertex from $\check{\mathcal{G}}_{r}$, together with the edge incident with $v$, we obtain a graph with ($r-1$) vertices, $r$ edges and {\it one} face. If this graph contains also a vertex of degree 1, we proceed successively. The process stops after $L$ steps, resulting into a (connected, cellularly embedded) multigraph, say $\hat{\mathcal{G}}_{\rho}$. This graph admits $\rho=r\!-\!L$
vertices (each of degree $\geqslant 2$), $\rho+1$ edges and one face.\\

\noindent
Apparently\footnote{Assume $L=r-1$.Then $\hat{\mathcal{G}}_{\rho}, \rho =1,$ would be a connected subgraph of $\mathcal{G}$, with two edges and one vertex; this contradicts the fact that $\mathcal{G}$ has no loops. Compare the forthcoming Definition \ref{ND8.5}.
} we have:  $
L<r\!-\!1 \text{ and thus }2 \leqslant \rho \leqslant r.$ In particular:\\ 
if $r=2$, then $\rho=2 \text{ and } L=0$.  \\
if $r=3$, then $\rho=3 \text{ and } L=0$, or $\rho=2 \text{ and } L=1$; see also Fig.\ref{Figure20}.\\

\noindent
From Subsection 1.3, Fig.\ref{Figure9.12},
it follows that
$\mathcal{G}_{2}$  is unique (up to equivalency).

\noindent
From the forthcoming Corollary \ref{NC8.2} it follows that also $\check{\mathcal{G}}_{2}$  is unique. However, a graph $\check{\mathcal{G}}_{r}, r>2,$ is not uniquely determined by $\mathcal{G}_{r}$, as will be clear from Fig.\ref{Figure20}, where $\mathcal{G}_{3}$ is a Newton graph
(cf. Subsubsection 1.2.3 or Fig.\ref{Figure9.11}(iii)).

\begin{lemma}
\label{NL8.1}
For the graphs $\hat{\mathcal{G}}_{\rho}$ we have:\\
\begin{enumerate}
\item[$(a_{1})$] Either, two vertices are of degree 3, and all other vertices of degree 2, or
\end{enumerate}
\begin{enumerate}
\item[$(a_{2})$] 
one vertex is of degree 4, and all other vertices of degree 2. 
\item[$(b_{})$] 
There is a closed, clockwise oriented facial walk, say $w$, of length $2(\rho+1$) such that, traversing $w$, each
vertex $v$ shows up precisely deg$ (v)$ times. Moreover, $w$ is divided into subwalks $W_{1}$, $W_{2}, \cdots ,$ connecting vertices of degree $>2$
that, apart from these begin- and endpoints, contain only-if any- vertices of degree 2. 
\end{enumerate}
\begin{enumerate}
\item[$(c_{1})$] 
If $e_{1}e_{2} \cdots e_{s}$
is a walk of type $W_{i}$, then also $
W_{i}^{-1}:=e_{s}^{-1} \cdots  e_{2}^{-1}e_{1}^{-1},
$
where $e$ and $e^{-1}$ stand for the same edge, but with opposite orientation.
\item[$(c_{2})$]
 The subwalks $W_{i}$ and $W_{i}^{-1}$ are not consecutive in $w$.
\item[$(c_{3})$]
In Case $(a_{1})$, there are precisely 6 subwalks of type $W_{i}$, each of them connecting different vertices (of degree 3). 
In fact there holds: $
w=W_{1}W_{2}W_{3}W_{1}^{-1}W_{2}^{-1}W_{3}^{-1}.
$
\item[$(c_{4})$]
In Case $(a_{2})$,  there are precisely 4 subwalks of type $W_{i}$ , each of them containing
at least one vertex of degree 2. In fact we have: $
w=W_{1}W_{2}W_{1}^{-1}W_{2}^{-1}.
$
\end{enumerate}
\end{lemma}

\begin{proof}
ad ($a$): Each edge of $\hat{\mathcal{G}}_{\rho}$ contributes precisely twice to the set $\{\text{deg}(v), v \in V(\hat{\mathcal{G}}_{\rho})\}$. 
It follows:
\begin{equation}
\label{vgl30}
\sum_{\text{all $\hat{\mathcal{G}}_{\rho}$-vertices $v$}} \text{ deg}(v)=2(\rho +1).
\end{equation}
Put $k_{i}\!=\! \sharp \{ \text{vertices of degree } i \}, i=1,2,3, \cdots $. 
Then (\ref{vgl30}) yields:
$$
2k_{2}\!+\!3k_{3}\!+\!4k_{4}\!+\!5k_{5}\!+ \!\cdots =2(k_{2} \!+\!k_{3}\!+\!k_{4}\!+\!k_{5} \!+\! \cdots \!+\!1) (=2(\rho +1)).
$$
Note that one uses here also that $\hat{\mathcal{G}}_{\rho}$ has $\rho$ vertices and all these vertices have degree $\geqslant 2$.
Thus, either $k_{2}=\rho\! -\!2, k_{3}=2,$ $k_{i} =0$ if $i \neq 2, 3,$  or  $k_{2}=\rho \!-\!1$, $k_{4}=1$, $k_{i} =0$ if $i \neq 2, 4$.

ad ($b$): The geometrical dual $(\hat{\mathcal{G}}_{\rho})^{*}$ of $\hat{\mathcal{G}}_{\rho}$ has only one vertex. So, all edges of $(\hat{\mathcal{G}}_{\rho})^{*}$ are loops. Hence, in the {\it facial} walk $w$ of $\hat{\mathcal{G}}_{\rho}$, each edge shows up precisely twice (with opposite orientation) cf. Subsubsection 1.2.3. Thus $w$ has length $2(\rho+1)$. By the face traversal procedure, each facial sector of $\hat{\mathcal{G}}_{\rho}$ is encountered once and -at a vertex $v$- there are deg($v$) many of such sectors. Application of ($a_{1}$) and ($a_{2}$) yields the second part of the assertion.

ad ($c_{1}$):
Let $e_{1}ve_{2}$
be a 
subwalk of $w$ with deg$(v)=2$. Both $e_{1}^{-1}$ and $e_{2}^{-1}$ occur precisely once in $w$, and $e_{2}^{-1}ve_{1}^{-1}$
is a subwalk of $w$. 

ad ($c_{2}$): If the subwalk $W=e_{1\!}e_{2} \cdots  e_{s}$ 
and its inverse are consecutive, then -by the face traversal procedure- 
$e_{1}^{-1}v_{1} e_{1}$, or  $e_{s}v_{s}e_{s}^{-1}$
are subwalks of the facial walk $w$. In the first case, $v_{1}$ is both begin- and endpoint of $e_{1}$; in the second case, $v_{s}$ is both begin- and endpoint of $e_{s}$. This would imply that $\hat{\mathcal{G}}_{\rho}$ has a loop which is excluded by construction of this graph.

ad ($c_{3}$): Note that, traversing $w$ once, each of the two vertices of degree 3 is encountered thrice. Suppose that the begin en and points of one of the subwalks $W_{1},W_{2}, \cdots $, say $W_{1},$ coincide. Then, this also holds for the subwalk $W_{1}^{-1}$ being-by ($c_{2}$)-not adjacent to $W_{1}$. So, traversing $w$ once, this common begin/endpoint is encountered at least four times; in contradiction with our assumption. The assertion is an easy consequence of ($a_{1}$) and ($c_{2}$).

ad ($c_{4}$): Traversing $w$ once, the vertex of degree 4 is encountered four times. In view of ($a_{2}$) and ($c_{2}$), we find four closed subwalks of $w$  namely: $W_{1}, W_{1}^{-1}$ (not adjacent) and $W_{2}, W_{2}^{-1}$ (not adjacent). Finally we note that each of these subwalks must contain at least one vertex of degree 2 (since $\hat{\mathcal{G}}_{\rho}$ has no loops).
\end{proof}

\begin{figure}[h!]
\begin{center}
\includegraphics[scale=0.70]{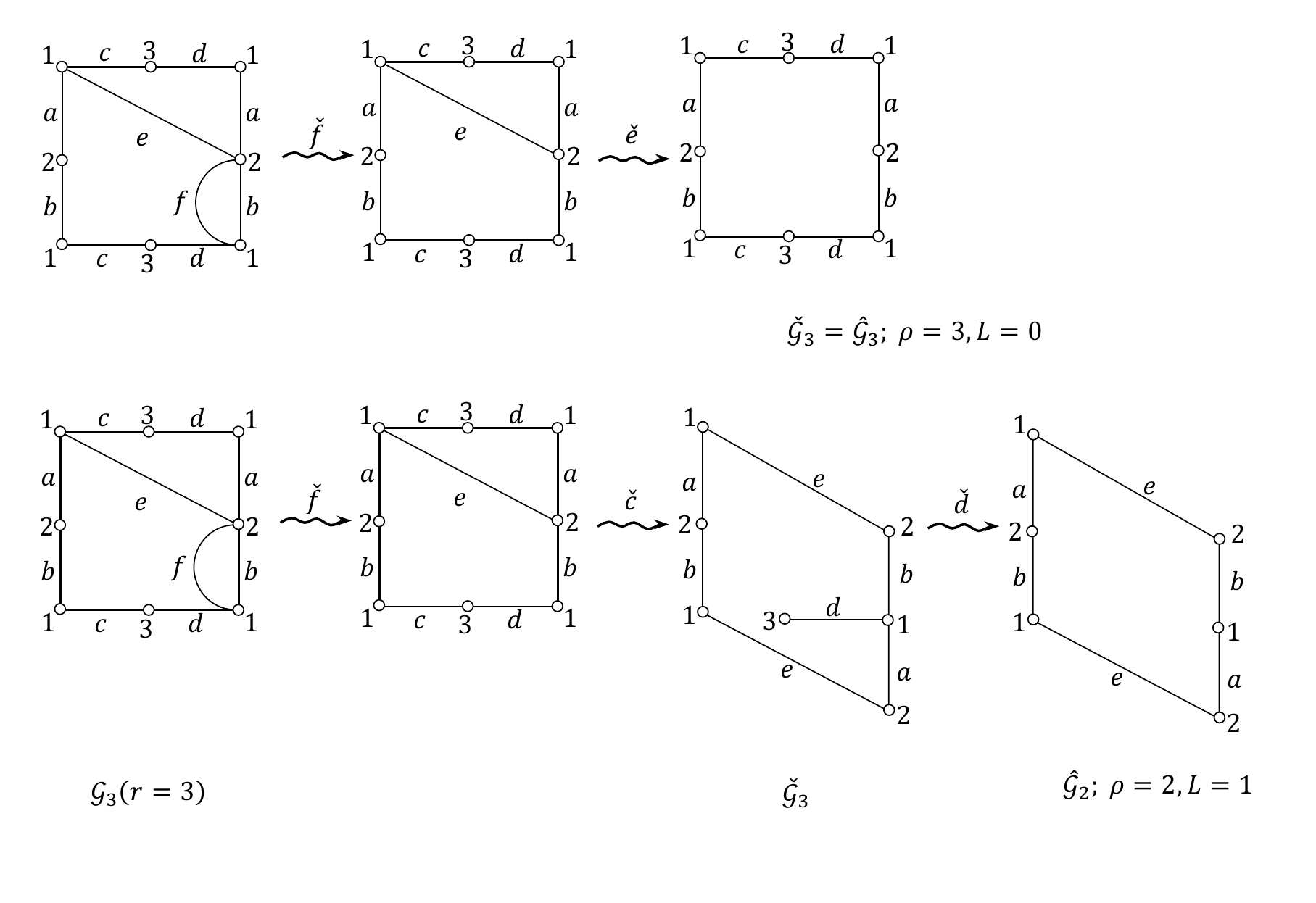}
\caption{\label{Figure20}Some graphs $\mathcal{G}_{3}$, $\check{\mathcal{G}}_{3}$ and $\hat{\mathcal{G}}_{\rho}$.}
\end{center}
\end{figure}

\begin{figure}[h!]
\begin{center}
\includegraphics[scale=0.76]{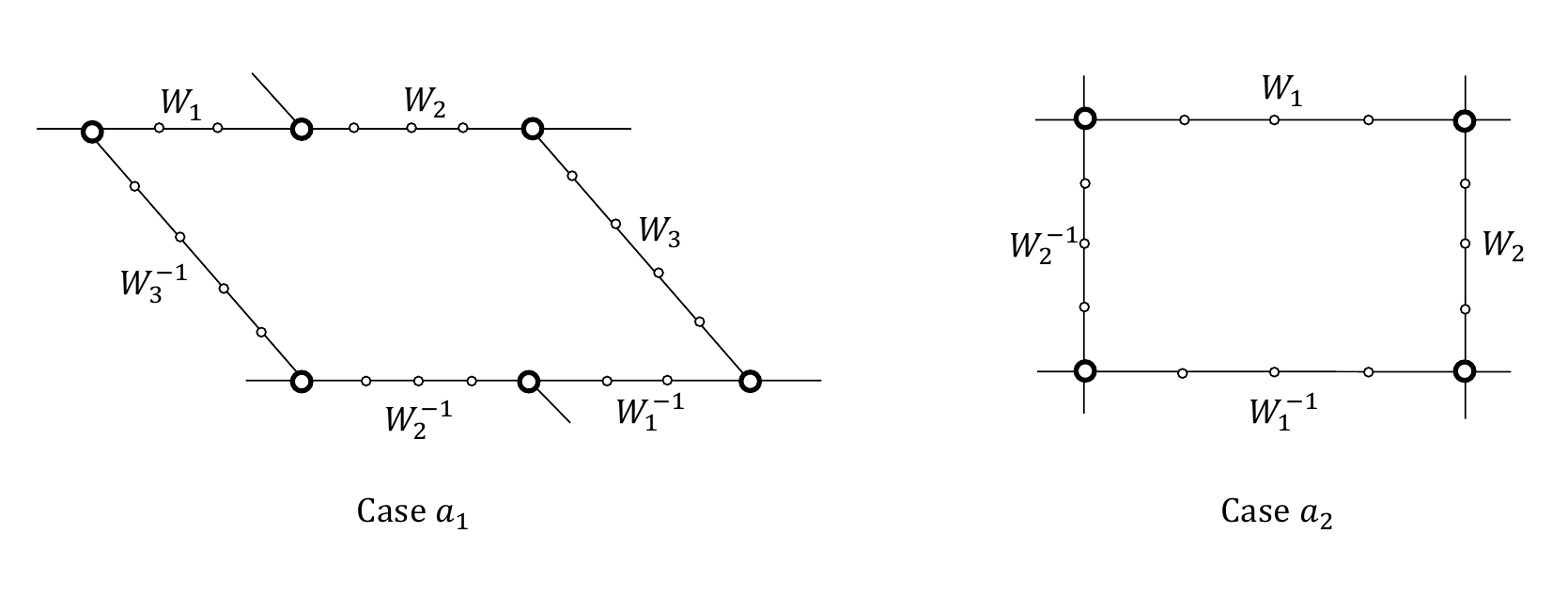}
\caption{\label{Figure21}The graphs $\hat{\mathcal{G}}_{\rho}$.}
\end{center}
\end{figure}

An analysis of its rotation system learns that $\hat{\mathcal{G}}_{\rho}$ is determined by its facial walk $w$, and thus also by the subwalks 
$W_{1}W_{2}W_{3}$ (in Case $a_{1}$) or $W_{1}W_{2}$ (in Case $a_{2}$). In fact, only the length of the subwalks $W_{i} $ matters. 
\begin{corollary}
\label{NC8.2}
The graphs $\hat{\mathcal{G}}_{2}$ and $\hat{\mathcal{G}}_{3}$ can be described as follows:
\begin{itemize}
\item 
By Lemma \ref{NL8.1} it follows that $\hat{\mathcal{G}}_{2}$ does not have a vertex of degree 4. So, $\hat{\mathcal{G}}_{2}$ is of the form as depicted in Fig.\ref{Figure21}-$a_{1}$, where each subwalk $W_{i}$ admits only one edge. Hence, there is-up to equivalency- only one possibility for $\hat{\mathcal{G}}_{2}$. Compare also Fig.\ref{Figure9.12}. 
\item
It is easily verified that -in Fig.\ref{Figure22}- each graph (on solid and dotted edges) is 
a Newton graph $($cf. Subsubsection 1.2.3 $)$.
Hence, in case $\rho=3$, both alternatives in Lemma \ref{NL8.1}-$(a)$ occur. An analysis of their rotation systems learns that the three graphs with only solid edges 
in Fig.\ref{Figure22}-$(i)-(iii)$ are equivalent, but not equivalent with the graph on solid edges in Fig.\ref{Figure22}-$(iv)$. In a similar way it can be proved that the graphs  in Fig.\ref{Figure22} expose all possibilities(up to equivalency) for $\hat{\mathcal{G}}_{3}$.
\end{itemize}
\end{corollary}
\begin{definition}
\label{D2.1}
{\it Pseudo Newton graphs}\\
Cellularly embedded toroidal graphs, obtained from $\mathcal{G}_{r}$ by deleting edges and vertices in the way as described above, are called {\it pseudo Newton graphs} (of order $r$).
\end{definition}
Apparently, a pseudo Newton graph is not a Newton graph by itself. Replacing (in the inverse order) into $\hat{\mathcal{G}}_{\rho}, \rho=r-L,$ 
the edges and vertices that we have deleted from $\mathcal{G}_{r}$ , we regain subsequently $\check{\mathcal{G}}_{\rho}$ and $\mathcal{G}_{r}$. A pseudo Newton graph of order $r$ has either one face with angles summing up to the number ($\rho^{'}$) of its vertices and ($\rho^{'}+1$)
edges ($2 \leqslant \rho \leqslant \rho^{'} \leqslant r$) or one face with angles summing up to $r', r-r'$ faces with angles summing up to 1 and altogether $2r-r'+1$ edges ($2 \leqslant r, 1<r' <r$). 

\begin{remark}
\label{NR8.4}
If we delete from $\hat{\mathcal{G}}_{\rho}$ an arbitrary edge, the resulting graph remains connected, but the Ôalternating sum of vertices, edges and faceÕ equals +1. Thus one obtains a graph that is not cellularly embedded. 
\end{remark}

\begin{figure}[h]
\begin{center}
\includegraphics[scale=0.76]{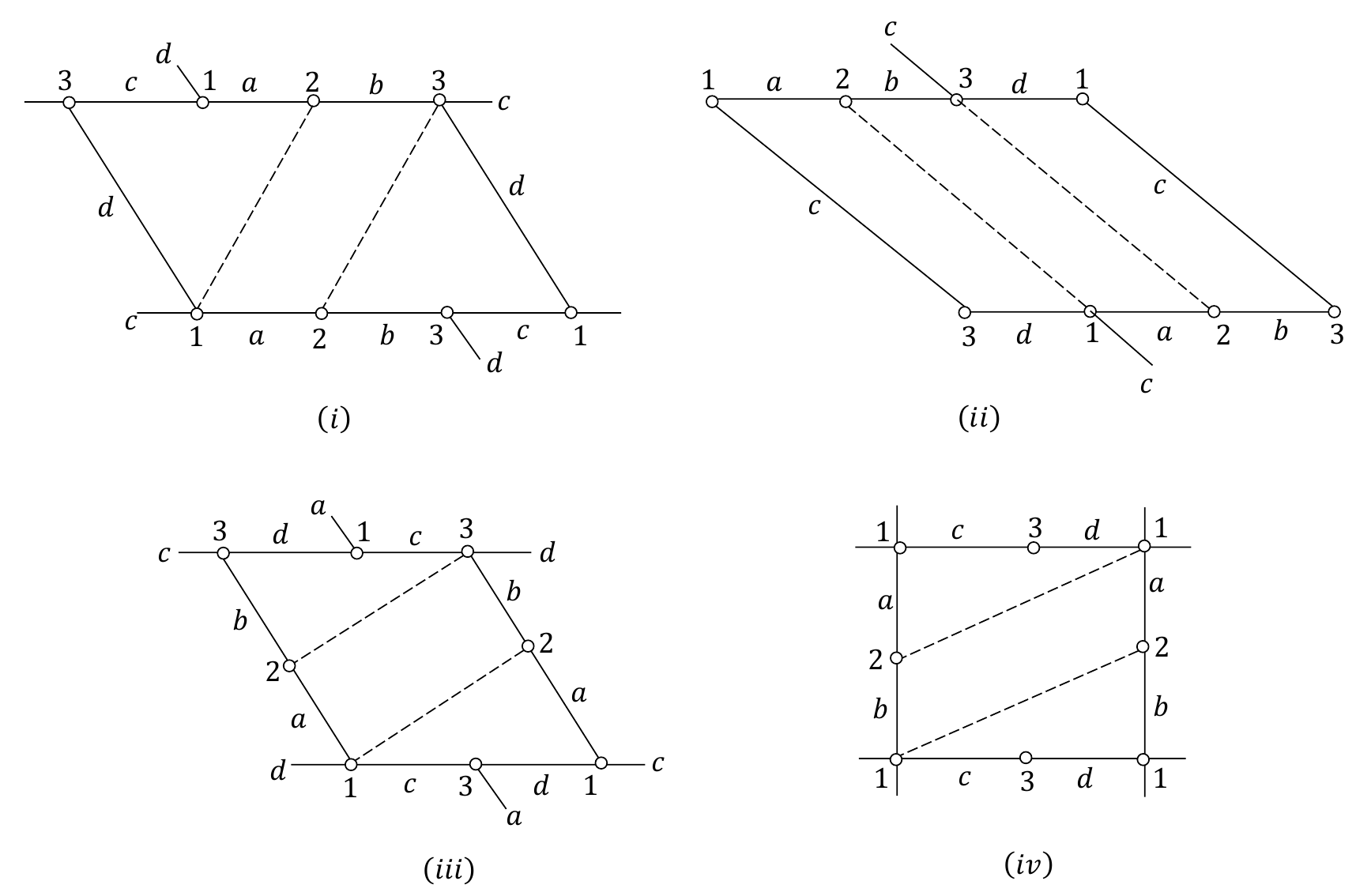}
\caption{\label{Figure22}All possible graphs $\hat{\mathcal{G}}_{3}$.}
\end{center}
\end{figure}

\begin{definition}
\label{ND8.5}
A {\it Nuclear Newton graph} is a cellularly embedded graph in $T$ with one vertex and two edges.
\end{definition}

\noindent
Apparently 
a Nuclear Newton graph is connected and admits one face and two loops. In particular such a graph has a trivial rotation system. Hence, all nuclear Newton graphs are topologically equivalent and since they expose the same structure as the pseudo Newton graphs $\hat{\mathcal{G}}_{\rho}$, they will be denoted by $\hat{\mathcal{G}}_{1}$. Note that 
a nuclear Newton graph fulfils the {\it A-property} (but certainly not the {\it E-property}). Consequently, a graph of the type $\hat{\mathcal{G}}_{1}$ is neither a Newton graph, nor  equivalent with a graph $\mathcal{G}(f), f \in \check{E}_{r}$. Nevertheless, nuclear Newton graphs will play an important role because, in a certain sense, they ``generate'' certain structurally stable Newton flows. This will be explained in the sequel.

\section{Nuclear elliptic Newton flow}

Throughout this section, let $f$ be an elliptic function with -viewed to as to a function on $T=T(\Lambda(\omega_{1}, \omega_{2}))$
- only one zero and one pole, both of order $r, r\geqslant 2$. Our aim is to derive the result on the corresponding (so-called nuclear) Newton flow $\overline{\overline{\mathcal{N}} }(f)$ that was already announced in \cite{HT1}, 
Remark 5.8. To be more precise:
\begin{align*}
&\text{`` All nuclear Newton flows -of any order $r$- are conjugate
, in particular each of them}\\
&\text{ has precisely two saddles (simple) and there are no saddle connections''.}
\end{align*}

\noindent
When studying -up to conjugacy-
the flow $\overline{\overline{\mathcal{N}} }(f)$, we may assume 
(cf. Subsubsection 1.1.3)
that 
$\omega_{1}=1, \omega_{2}=i,$ thus $\Lambda=\Lambda_{1,i}$.
In particular, the period pair 
$(1,i)$
is reduced. 
We represent $f$ (and thus  $\overline{\overline{\mathcal{N}} }(f)$), by the $\Lambda$-classes $[a], [b]$
, where $a$ resp. $b$
stands for the zero, resp. pole, for $f$, situated in the period parallelogram $P(=P_{1,i})$. 
Due to (2), (3)
we have:
\begin{equation*}
b=  a+\frac{\lambda_{0}}{r}.
\end{equation*}
We may assume that $a$, $b$ 
are not on the boundary $\partial P$ of $P$.
Since the period pair $(1,i)$
is reduced, the images under $f$ of the $P$-sides $\gamma_{1}$ and $\gamma_{2}$
are closed Jordan curves 
(use the explicit formula for $\lambda^{0}$ as presented in Footnote 3).
From this, we find 
that the winding numbers $\eta(f(\gamma_{1}))$ and $\eta(f(\gamma_{2}))$
can -a priori- only take the values -1, 0 or +1. The combination 
$(\eta(f(\gamma_{1})),\eta(f(\gamma_{2})))=(0, 0)$ is impossible (because  $a \neq  b$).
The remaining combinations lead, for each value of $r=2,3, \cdots $, to eight different values for $b$ each of which giving rise, together with $a$, to eight pairs of classes $\mathrm {mod} \, \Lambda$ that fulfil (\ref{vgl9}), determining flows in $N_{r}(\Lambda)$, compare Fig. \ref{Figure25}, where we assumed -under a suitable translation of $P$-that  $a=0$.

\begin{figure}[h]
\begin{center}
\includegraphics[scale=0.6]{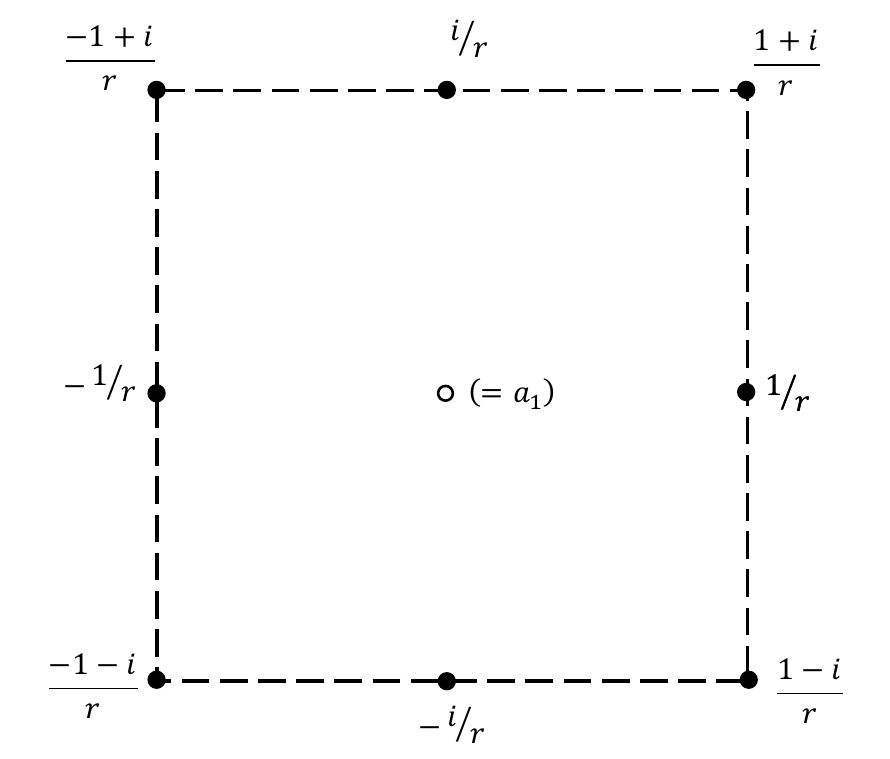}
\caption{\label{Figure25} Eight pairs $(a,b)$ determining a priori possible nuclear flows in $N_{r}(\Lambda_{1,i}).$}
\end{center}
\end{figure}

Note that the derivative $f'$ is elliptic of order $r+1$. Since there is on $P$ only one zero for $f$ (of order $r$), 
the function $f$ has two critical points, i.e., saddles for $\overline{\overline{\mathcal{N}} }(f)$, counted by multiplicity.

The eight pairs $(a, b)$ that possibly determine a nuclear Newton flow are subdivided into two classes, each containing four configurations $(a, b)$: (see Fig.\ref{IVFigure9})\\

\noindent
{\bf Class 1:} $a=0$, $b$ on a side of the period square $P $.\\
{\bf Class 2:} $a=0$, $b$ on a diagonal of the period square $P $, but not on $\partial P$.\\

\noindent
Apparently, two nuclear Newton flows represented by configurations in the same class are related by a {\it unimodular} transformation on the period pair $(1,i)$, and are thus conjugate, see Subsubsection 1.1.3. 
So, it is enough to study nuclear Newton flows, possibly represented by $(0, \frac{1}{r})$ or $(0, \frac{1+i}{r})$.\\

\noindent
\underline{The configuration $(a, b), a=0, b=\frac{1}{r}:$}\\
The line $\ell_{1}$ between 0 and 1, and the line $\ell_{2}$ between $\frac{1}{i}$ and $1+ \frac{1}{i}$, are axes of mirror symmetry with respect to this configuration (cf. Fig.\ref{IVFigure10}-(i)). By the aid of this symmetry and using the double periodicity of the supposed flow
, it is easily proved that this configuration can not give rise to a desired nuclear Newton flow.\\

\noindent
\underline{The configuration $(a, b), a=0, b=\frac{1+i}{r}:$}\\
The line $\ell$ between 0 and $\frac{1+i}{2}$ is an axis of mirror symmetry with respect to this configuration (cf. Fig.\ref{IVFigure10}-(ii)). So the two saddles of the possible nuclear Newton flow are situated either on the diagonal of $P$ through $\frac{1+i}{r}$, or not on this diagonal but symmetric with respect to $\ell_{1}$. The first possibility can be ruled out (by the aid of the symmetry w.r.t. $\ell$ and using he double periodicity of the supposed flow). 
So it remains to analyze the second possibility. (Note that only in the case where $r=2$, also the second diagonal of $P$ yields an axis of mirror symmetry).\\

\begin{figure}[h]
\begin{center}
\includegraphics[scale=0.5]{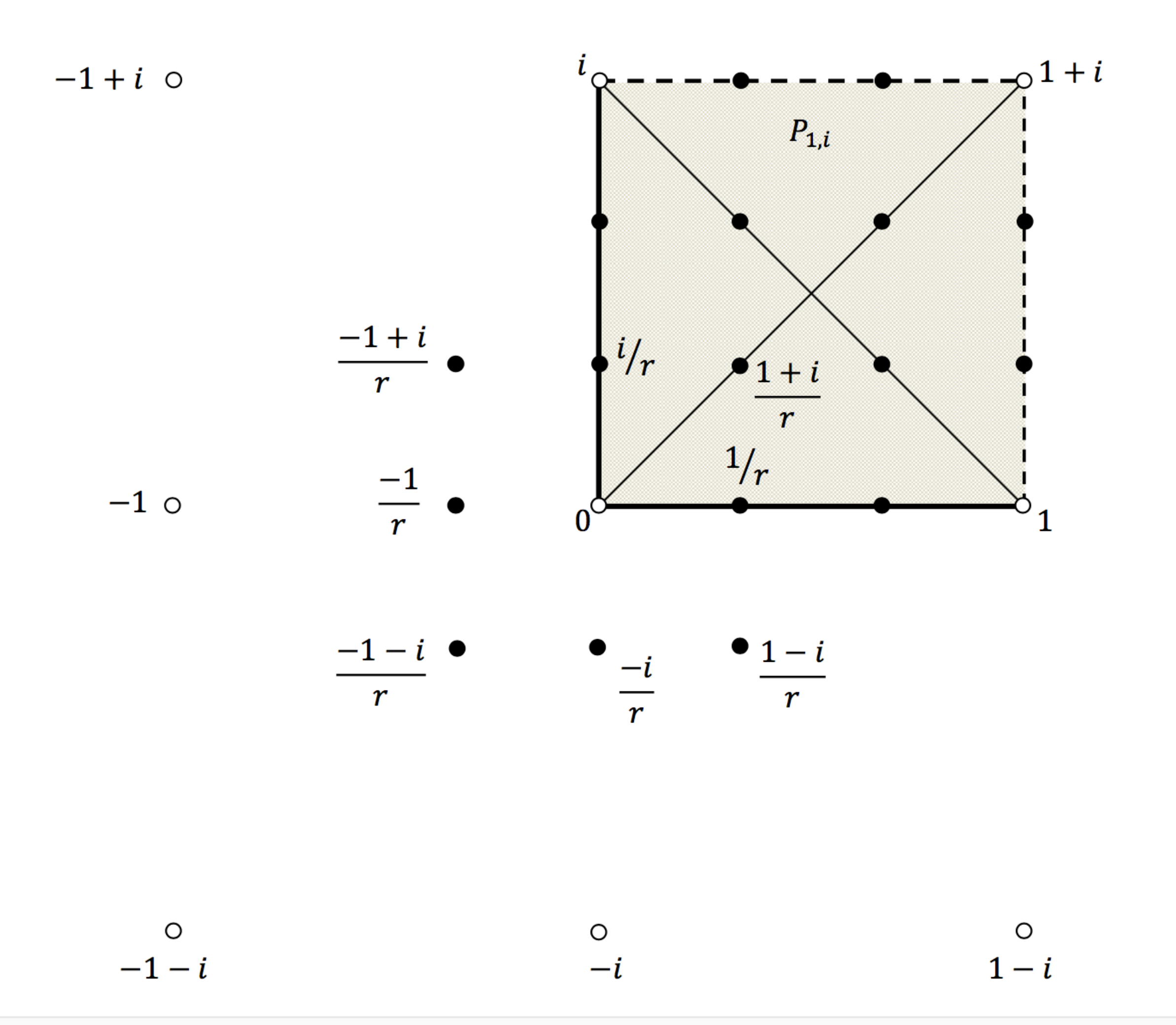}
\caption{\label{IVFigure9} The two classes of pairs $(a,b)=(0, \bullet)$ on $P_{1,i}$.} 
\end{center}
\end{figure}

We focus on Fig.\ref{IVFigure11}, where the only relevant configuration determining a (planar) flow $\overline{\overline{\mathcal{N}} }(f)$, is depicted. By symmetry, the $\ell$-segments between 0 and $\frac{1+i}{r}$, and between $\frac{1+i}{r}$ and $(1+i)$ are $\overline{\overline{\mathcal{N}} }(f)$-trajectories connecting the pole $\frac{1+i}{r}$, with the zeros $0$ and $(1+i)$. Since on the $\overline{\overline{\mathcal{N}} }(f)$-trajectories the arg($f$) values are constant, we may arrange the argument function on $\mathbb{C}$ such that on the segment between $\frac{1+i}{r}$ and $1+i$ we have arg($f$)$=0$. We put arg$f$($\sigma_{1}$)$=\alpha$, thus $0<\alpha<1$ and arg $f$($\sigma_{2}$)$=-\alpha$. Note that at the zero / pole for $f$, each value of arg($f$) appears $r$ times on equally distributed incoming (outgoing) $\overline{\overline{\mathcal{N}} }(f)$-trajectories. By the aid of this observation, together with the symmetry and periodicity of $f$, we find out that the phase portrait of $\overline{\overline{\mathcal{N}} }(f)$ is as depicted in Fig.\ref{IVFigure11}, where box stands for the (constant values of arg($f$) on the unstable manifolds of $\overline{\overline{\mathcal{N}} }(f)$. In particular, there are no saddle connections.\\

\begin{figure}[h]
\begin{center}
\includegraphics[scale=0.3]{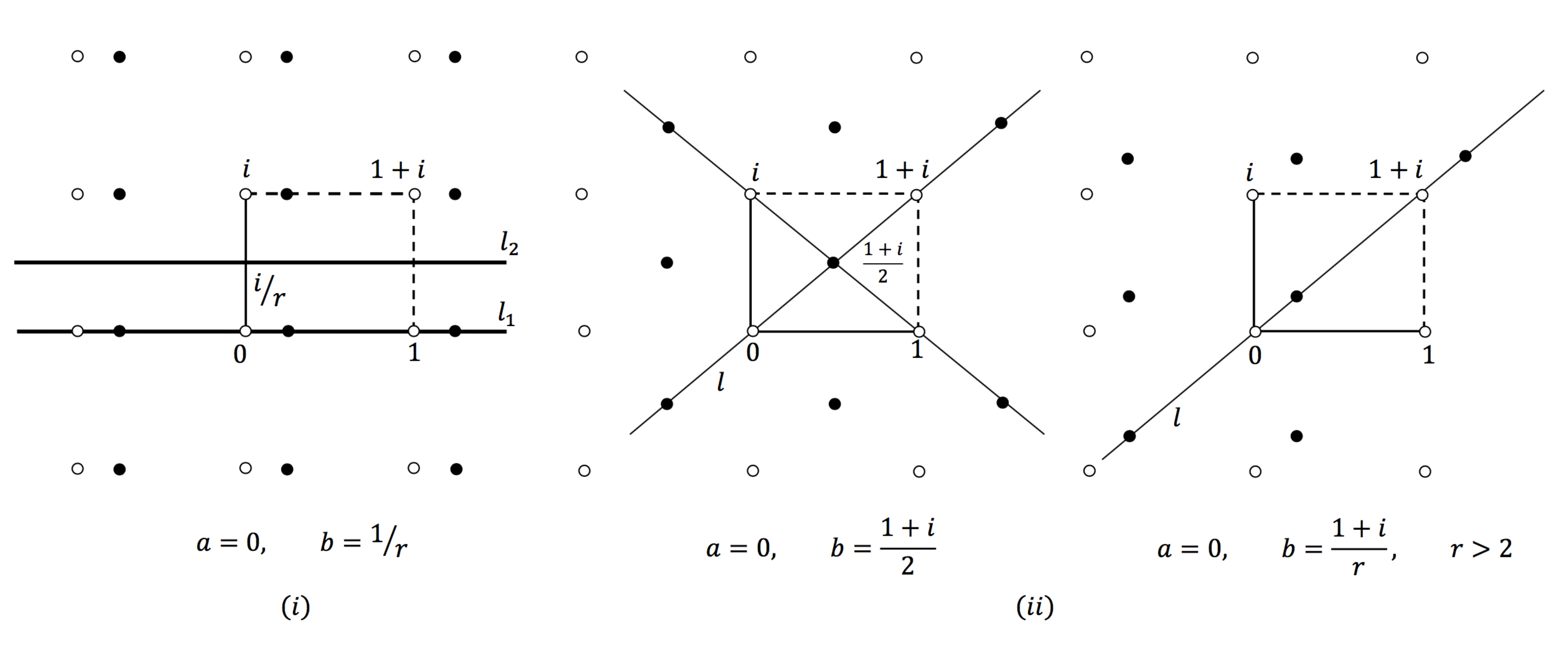}
\caption{\label{IVFigure10} Axes of window symmetry for the configuration $(a,b) (=(0, \bullet))$.}
\end{center}
\end{figure}

\begin{figure}[h]
\begin{center}
\includegraphics[scale=0.4]{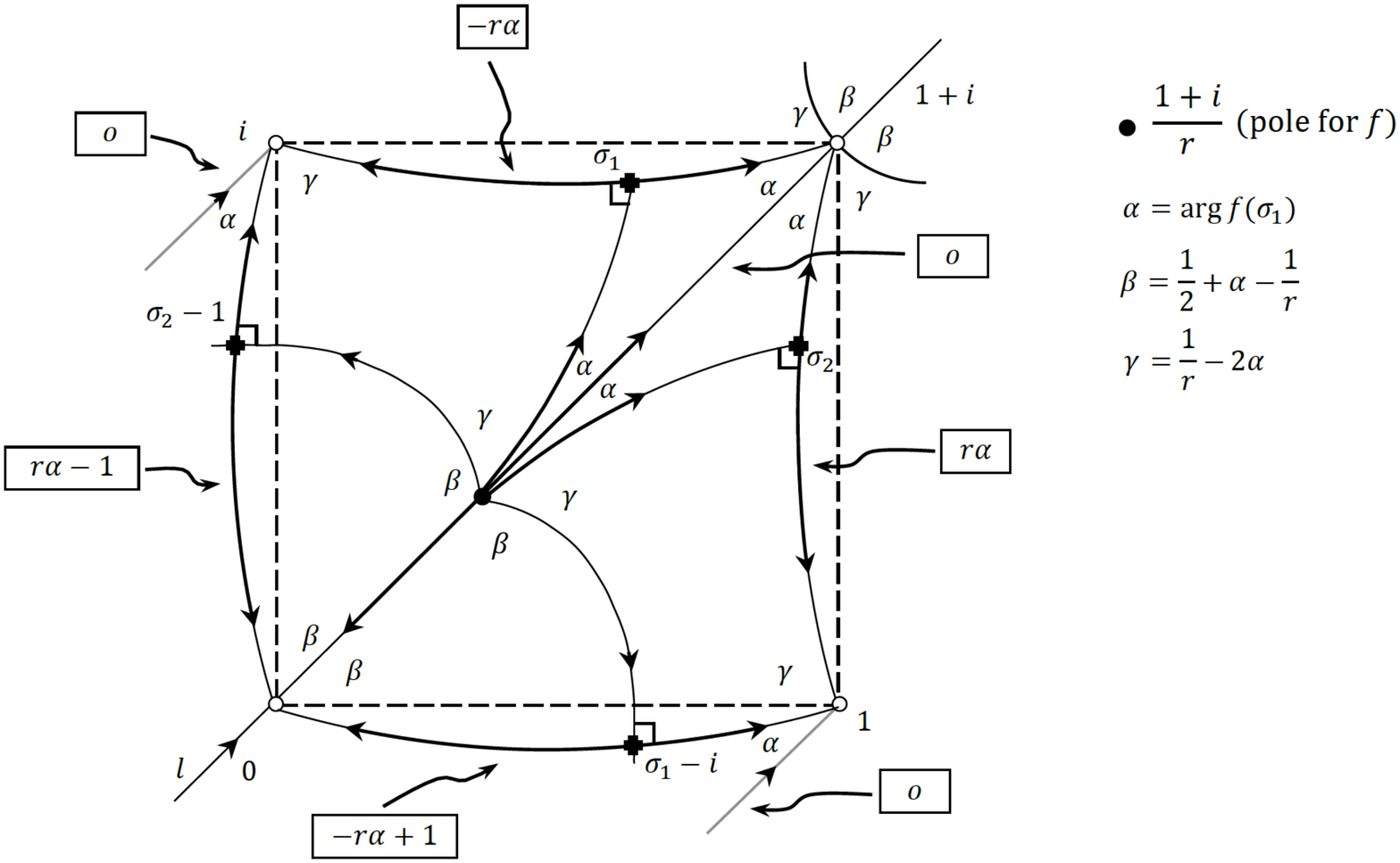}
\caption{\label{IVFigure11}The phase portrait of a nuclear Newton flow of order $r>2$.}
\end{center}
\end{figure}

\begin{remark}
\label{R3.1}
{\it Canonical form of the phase portrait of a nuclear Newton flow}\\
The qualitative features of the phase portrait of $\overline{\overline{\mathcal{N}} }(f)$ in Fig.\ref{IVFigure11} rely on the values of $r$ and $\alpha$.
Put $\beta=\frac{1}{2}+\alpha -\frac{1}{r}$ and $\gamma=\frac{1}{r}- 2 \alpha$. Since all angles $\alpha, \beta \text{ and }\gamma$ are strictly positive, we have $0 < \alpha < \frac{1}{2r}$.
Note that if $r=2$, by symmetry w.r.t. both the diagonals of $P$, we have $\alpha=\beta=\frac{1}{8};\gamma=\frac{1}{4}$
, see Fig.\ref{IVFigure12}; in this case $\overline{\overline{\mathcal{N}} }(f)$ is just $\overline{\overline{\mathcal{N}} }(\wp)$, with $\wp$ the Weierstrass'  $\wp$-function (lemniscate case, cf. \cite{A/S}).
\end{remark}

\begin{figure}[h]
\begin{center}
\includegraphics[scale=0.3]{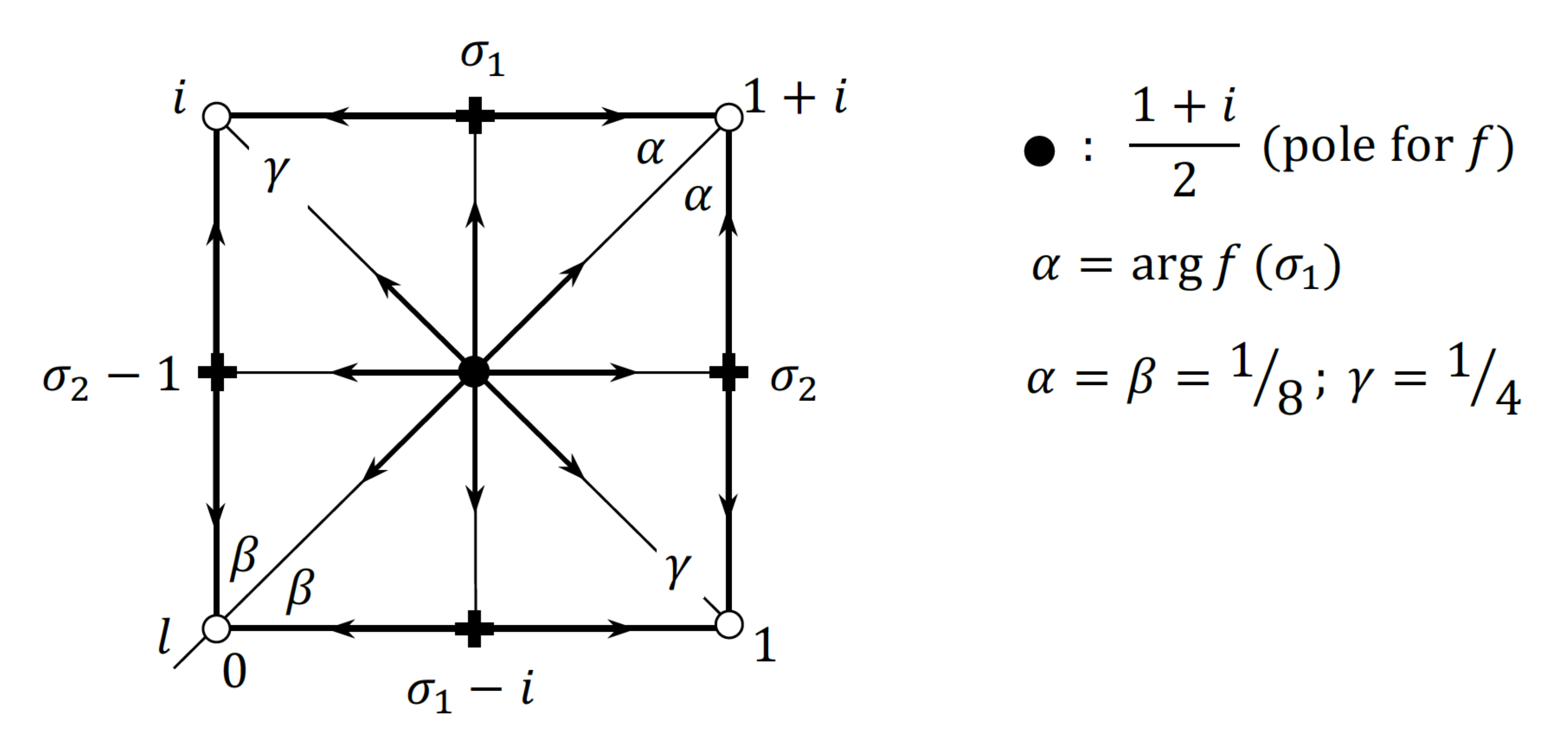}
\caption{\label{IVFigure12} The canonical nuclear Newton flow of order $r=2$.}
\end{center}
\end{figure}

Altogether we conclude:
\begin{lemma}
\label{L9.1}
All nuclear elliptic Newton flows of the same order $r$ are mutually conjugate.
\end{lemma}

For $f$ an elliptic function of order $r$ with - on $T$ - only one zero and pole we now define:
\begin{definition}
\label{D3.3}
$\mathcal{H}_{r}(f)$ is the graph on $T$ with as vertex, edges and face respectively:\\
- the zero for $f$ on $T$ (as atractor for $\overline{\overline{\mathcal{N}} }(f)$);\\
- the unstable manifolds for $\overline{\overline{\mathcal{N}} }(f)$ at the two critical points for $f$;\\
- the basin of repulsion for $\overline{\overline{\mathcal{N}} }(f)$ of a pole for $f$ on $T$ (as a repellor for $\overline{\overline{\mathcal{N}} }(f)$).
\end{definition}
From Fig.11  it is evident that $\mathcal{H}_{r}(f)$ is a cellularly embedded pseudo graph (loops and multiple edges permitted). This graph is referred to as to the {\it nuclear Newton graph} for $\overline{\overline{\mathcal{N}} }(f)$. By Lemma \ref{L9.1},  the graphs $\mathcal{H}_{r}(f)$, are - up to equivalency - unique, and will be denoted by $\mathcal{H}_{r}$ (compare the comment on Definition \ref{ND8.5}).

If $a$, $b$ (both in $P(=P_{1,i})$) are of Class 2 (i.e., the configuration $(a, b)$ determines a nuclear flow with $a$ and $b$ as zero resp. pole of order $r$), we introduce the doubly periodic functions:

\begin{equation}
\label{Nvgl30}
\Psi_{a}(z)=\sqrt{\sum_{\omega \in \Lambda}|z-a-\omega|^{-(4r-4)}}; \;\;\Psi_{b}(z)=\sqrt{\sum_{\omega \in \Lambda}|z-b-\omega|^{-(4r-4)}},
\end{equation}
where the summation takes place over all points in lattice $\Lambda(=\Lambda_{1,i})$. 

We define the planar flow $\underline{\mathcal{N}}(f)$
by:
\begin{equation}
\label{pNf}
\frac{dz}{dt}=-\Psi_{a}(z)\Psi_{b}(z)(1+|f(z)|^{4})^{-1}\overline{f^{'}(z)}f(z)
\end{equation}
\begin{lemma}
\label{L9.5}
The flow $\underline{\mathcal{N}}(f)$
is smooth on $\mathbb{C}$ and exhibits the same phase portrait as $\overline{\mathcal{N}}(f)$, 
but, its attractors (at zeros for $f $)  and its repellors (at the poles for $f$) are all generic, i.e. of the hyperbolic type. 
\end{lemma}
\begin{proof}
Since $r\geqslant 2$ (thus $4r-4 \geqslant 4$), series of the type as under the square root in (\ref{Nvgl30}) are uniform convergent in each compact subset of 
$
\mathbb{C} \backslash (a+ \Lambda \cup b+\Lambda ).
$
From this, together with the smoothness of $\overline{\mathcal{N}}(f)$ on $\mathbb{C}$, it follows that $\underline{\mathcal{N}}(f)$
is smooth outside the union of $a+ \Lambda$ and $b+ \Lambda$.
Special attention should be paid to the lattice points. Here the smoothness of  $\underline{\mathcal{N}}(f)$ as well as the genericity of its attractors and repellors
follows by a careful (but straightforward) analysis of the local behaviour of $\underline{\mathcal{N}}(f)$ around these points; compare the explicit expression for  $\overline{\mathcal{N}}(f)$ in Footnote 1 and note that zeros and poles are of order $r$.
Since outside their equilibria $\underline{\mathcal{N}}(f)$
and $\overline{\mathcal{N}}(f)$ are equal -up to a strictly positive factor- their portraits coincide.
\end{proof}

\begin{corollary}
\label{C3.5}
All nuclear Newton flows of arbitrary order are mutually conjugate.
\end{corollary}
\begin{proof}
Let $\underline{N}(f)$ be arbitrary. Because all its equilibria are generic and there are no saddle connections, this flow is $C^{1}$-structurally stable. The embedded graph $\mathcal{H}_{r}(=\mathcal{H}_{r}(f))$, together with its geometrical dual $\mathcal{H}_{r}(f)^{*}$, forms the so-called distinguished graph that determines - up to an orientation preserving homeomorphism - the phase portrait of $\underline{N}(f)$ (cf. \cite{Peix1}, \cite{Peix2} and Subsubsection 1.2.3). This distinghuished graph is extremely simple, giving rise to only four distinghuished sets (see Fig.\ref{IVFigure11}). This holds for any flow of the type $\underline{N}(f)$. Now, application of Peixoto's classification theorem for $C^{1}$-structurally stable flows on $T$ yields the assertion.
\end{proof}

We end up with a comment on the nuclear Newton graph $\mathcal{H}(f_{\omega_{1}, \omega_{2}})$, ${\rm Im}\frac{\omega_{2}}{\omega_{1}} >0,$ where $(\omega_{1}, \omega_{2}), $ is related to the period pair $(1, i)$ by the unimodular transformation
$$
M=\left(
\begin{matrix}
p_{1}&q_{1} \\
p_{2}&q_{2}
\end{matrix}
\right), \;\; p_{1}q_{2}-p_{2}q_{1}=+
1.
$$
Thus $(p_{1}, p_{2})$ and $(q_{1}, q_{2})$ are co-prime, and
$$
\omega_{1}=p_{1}+p_{2}i, \;\; \omega_{2}=q_{1}+q_{2}i.
$$
Our aim is to describe $\mathcal{H}(f_{\omega_{1}, \omega_{2}})$ as a graph on the canonical torus $T(=T_{1,i})$. 
In view of Lemma \ref{L9.1}, the two edges of $\mathcal{H}(f_{\omega_{1}, \omega_{2}})$ are closed Jordan curves on $T$ , corresponding to the unstable manifolds of $\overline{\mathcal{N}}(f_{\omega_{1},\omega_{2}})$) at the two critical points for $f$ that are situated in the period parallelogram $P_{\omega_{1}, \omega_{2}}$. These unstable manifolds connect $a(=0)$ with $p_{1}+p_{2}i$, and $q_{1}+q_{2}i $ respectively. Hence, one of the $\mathcal{H}(f_{\omega_{1}, \omega_{2}}) $-edges wraps $p_{1}$-times around $T$ in the 
direction of the period 1 and $p_{2}$-times around $T $ in the 
direction of the period $i$, whereas the other edge wraps $q_{1}$-times around this torus in the 
$1$-direction respectively $q_{2}$-times in the 
$i$-direction. See also Fig.\ref{NFig49}, where we have chosen for $f$ the Weierstrass $\wp$-function (lemniscate case), i.e. $r=2$, $a=0$ and $\omega_{1}=3+i, \omega_{2}=2+i$. Compare also Fig.\ref{IVFigure12}, case $r=2$.
  
\begin{figure}[h!]
\begin{center}
\includegraphics[scale=0.7]{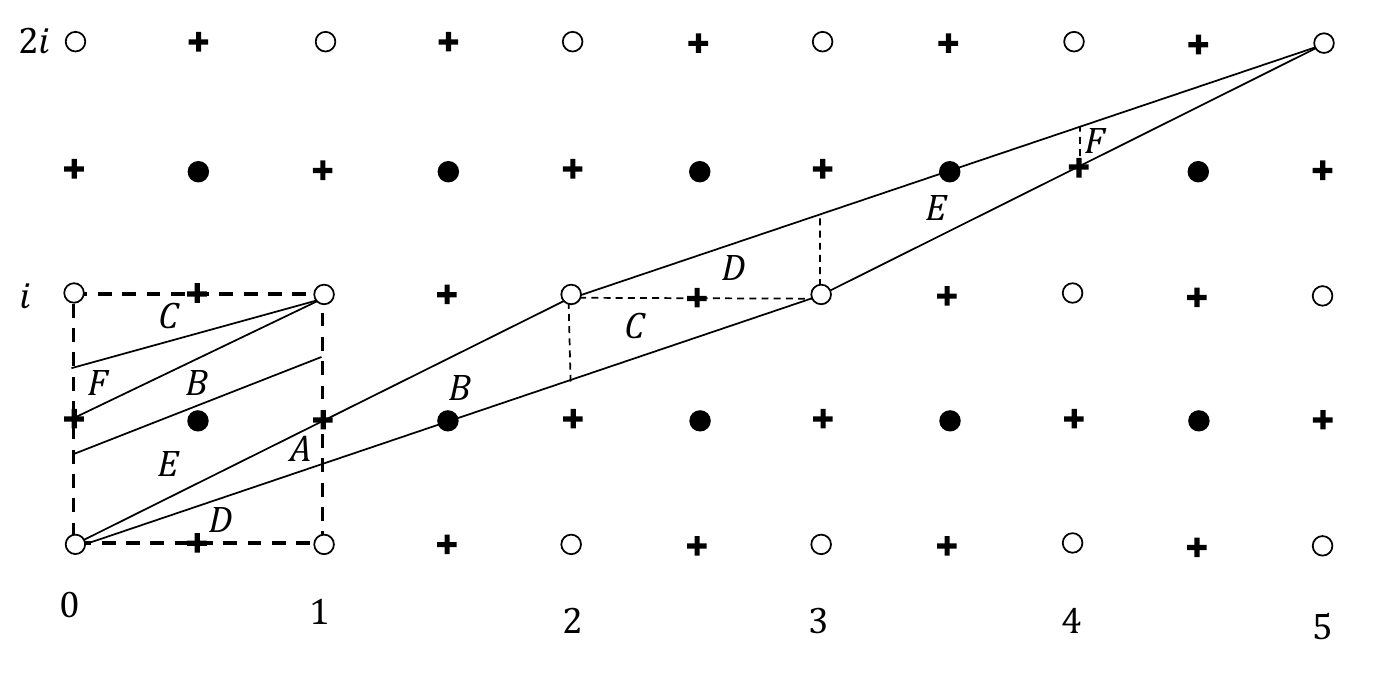}
\caption{\label{NFig49}The nuclear Newton graph $\mathcal{H}(\wp_{3+i,2+i})$ on the torus $T(=T_{1,i}).$} 
\end{center}
\end{figure}

\section{The bifurcation \& creation of elliptic Newton flows}

In this section we discuss the connection between pseudo Newton graphs and Newton flows.
In order not to blow up the size of our study, we focus - after a brief introduction - on the cases $r=2, 3$. However, even from these simplest cases we get some flavor of what we may expect when dealing with a more general approach.

\begin{figure}
\begin{center}
\includegraphics[width=5.5in]{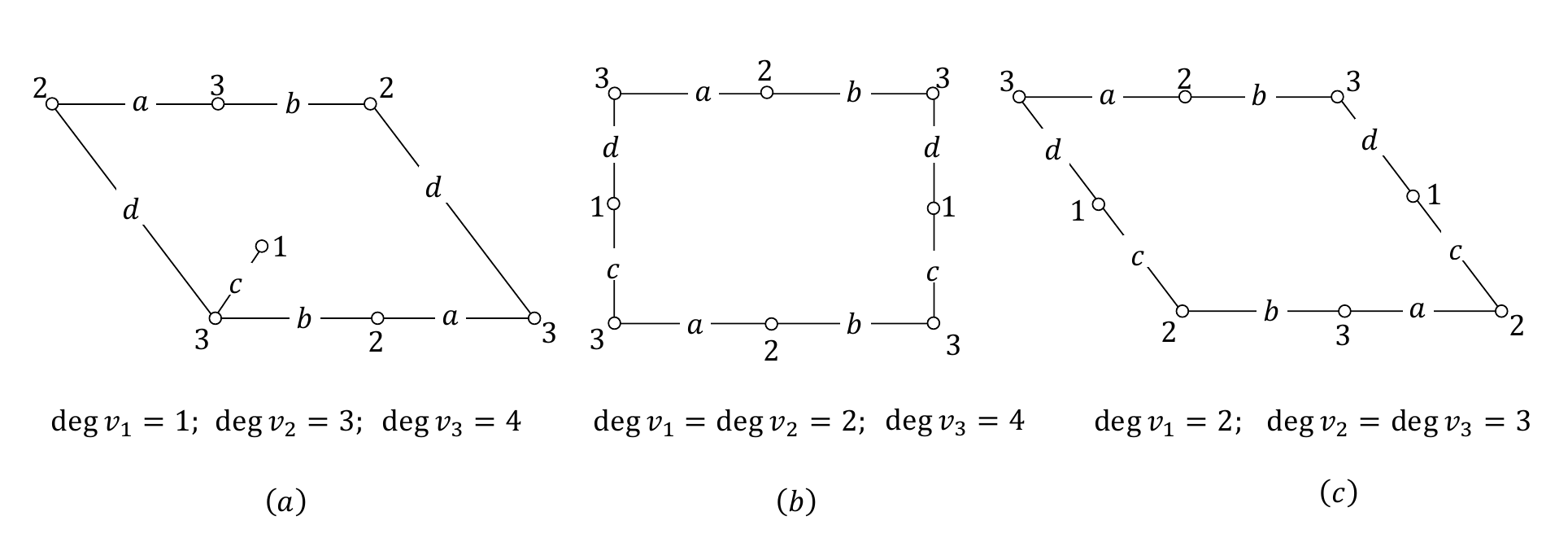}
\caption{\label{Figure29} The three different pseudo Newton graphs $\check{\mathcal{G}}_{3}$,
$\hat{\mathcal{G}}_{3}$.}
\end{center}
\end{figure} 

We consider functions $g \in E_{r}$ with $r$ simple zeros and only one pole (of order $r$); such functions exist,  compare Subsubsection 1.1.1. The set of all these functions is denoted by $ E_{r}^{1}$  and will be endowed with the relative topology induced by the topology $\tau_{0}$ on $E_{r}$. Since the derivative $g'$ of $g$ is elliptic of order $r+1$, the zeros for $g$ being simple, there are  $r+1$ critical points for $g$ (counted by multiplicity). 

We consider the set $N^{1}_{r}$ of all toroidal Newton flows $\overline{\overline{\mathcal{N}} }(g)$. Such a flow is $C^{1}$-structurally stable (thus also $\tau_{0}$-structurally stable) if and only if: (cf. subsection 1.1.4 and \cite{Peix1},  \cite{Peix2})
\begin{enumerate}
\item[1.] All saddles are simple (thus generic). 
\item[2.] There are no ``saddle connections''.
\item[3.] The repellor at the pole for $g$ is generic.
\end{enumerate}
In general none of these conditions is fulfilled. We overcome this complication as follows:

\noindent
{\bf ad 1.} Under suitably chosen - but arbitrarily small - perturbations of the zeros and poles of $g$, thereby preserving their multiplicities, $\overline{\overline{\mathcal{N}} }(g)$ turns into a Newton flow with only simple (thus $r+1$) saddles (cf. \cite{HT1}, Lemma 5.7, case  $A=r, B=1$).\\
{\bf ad 2.}  Possible saddle connections can be broken by Òadding to $g$ a  suitably chosen, but arbitrarily small constantÓ (cf.\cite{HT1}, proof of Theorem 5.6 (2)).\\
{\bf ad 3.}  With the aid of a suitably chosen additional damping factor to $\overline{\overline{\mathcal{N}} }(g)$, the pole of $g$ may be viewed to as generic for the resulting flow; compare the proof of Lemma 3.4. (Note that the simple zeros for $g$ yield already generic equilibriae).\\

This opens the possibility to adapt $g$ and $\overline{\overline{\mathcal{N}} }(g)$ in such a way that for Òalmost allÓ functions $g$ the flow $\overline{\overline{\mathcal{N}} }(g)$ is structurally stable (see Subsubsection 1.1.4, and Theorem 5.6 in  \cite{HT1}). More formally:\\

\noindent
The set $\underline{E}^{1}_{r}$ of functions $g$ in $ E_{r}^{1}$, with $\overline{\overline{\mathcal{N}} }(g)$ structurally stable, is $\tau_{0}$-open and -dense in $ E_{r}^{1}$.\\

From now on, we assume that $\overline{\overline{\mathcal{N}} }(g)$ is structurally stable and define the multi graph $\mathcal{G}_{r}(g)$ on $T$ as follows:\\
-  Vertices: $r$ zeros for $g$ (i.e., stable star nodes for $\overline{\overline{\mathcal{N}} }(g)$).\\
- Edges: $r+1$  unstable manifolds at the critical points for $g$ (orthogonal saddles for $\overline{\overline{\mathcal{N}} }(g)$).\\
- Face: the basin of repulsion of the unstable star node at the pole for $g$.\\
Note that $\mathcal{G}_{r}(g)$ has no loops (since the zeros for $g$ are simple).\\

It is easily seen that $\mathcal{G}_{r}(g)$ is cellularly embedded (cf. \cite{HT2}, proof of Lemma 2.9).\\

Because $\mathcal{G}_{r}(g)$ has only one face, the geometrical dual $\mathcal{G}_{r}(g)^{*}$ admits merely loops and the $\Pi$-walk for the $\mathcal{G}_{r}(g)$-face consists of $2(r+1)$ edges, each occurring twice, be it with opposite orientation; here the orientation on the $\Pi$-walk is induced by the anti-clockwise orientation on the embedded $\mathcal{G}_{r}(g)^{*}$-edges at the pole for $g$.
In the case where $\mathcal{G}_{r}(g)$ admits a vertex of degree 1, we delete this vertex together with the adjacent edge, resulting into a cellularly embedded graph on $r\!-\!1$ vertices, $r$ edges and only one face. If this graph has a vertex of degree 1, we repeat the procedure, and so on. The process stops after $L \, ( <\! r\!-\!1)$ steps, resulting into a connected, cellularly embedded muligraph of the type 
$\hat{\mathcal{G}}_{\rho},  \rho=r-L, 2 \leqslant \rho \leqslant r$.

Now, we raise the question whether the graphs obtained in this way are indeed pseudo Newtonian, i.e., do they originate from a Newton graph? And even so, can all pseudo Newton graphs be represented by elliptic Newton flows?

In the sequel we give an (affirmative) answer to these questions only in the cases $r=2$ and $r=3$.

\begin{lemma}
\label{L10.1}
If $r=2$ or $3$, then the graph $\mathcal{G}_{r}(g)$, $g \in \underline{E}^{1}_{r}$, is a pseudo Newton graph $\check{\mathcal{G}}_{r}$ .
\end{lemma}
\begin{proof}
Firstly, note that the proof of Lemma \ref{NL8.1} does not rely on the fact that $\hat{\mathcal{G}}_{\rho}$ originates from Newton graphs, but merely on the cellularity of $\hat{\mathcal{G}}_{\rho}$ in combination with the property that $\# \{\text{edges}\}=1+ \# \{\text{vertices}\}$.\\
\underline{Case $r=2$}: By Corollary \ref{NC8.2} of Lemma \ref{NL8.1} we know:  $\hat{\mathcal{G}}_{2}$ ($=\check{\mathcal{G}}_{2}$) is unique (up to equivalency). So, $\mathcal{G}_{2}(g)$ has the same topological type as $\hat{\mathcal{G}}_{2}$ 
and originates from a Newton graph (compare Fig.\ref{Figure9.12} and Fig.\ref{Figure21}$a_{1}$, where all subwalks $W_{i}$ admit only one edge).\\
\noindent
\underline{Case $r=3$}: If $\mathcal{G}_{3}(g)$ has a vertex $v_{1}$ of degree 1, the graph obtained by deleting $v_{1}$ together with the adjacent edge $c$ is a cellularly embedded graph in $T$ with two vertices and three edges and must $\hat{\mathcal{G}}_{2}$. So $\mathcal{G}_{3}(g)$ is of the form Fig.\ref{Figure29}(a).

If $\mathcal{G}_{3}(g)$ has no vertex  of degree 1, this graph is of type $\check{\mathcal{G}}_{3}$, and thus - by Corollary 2.2 -  either of the form as depicted in Fig.\ref{Figure29} (b), or Fig.\ref{Figure29} (c).

So, we find that $\mathcal{G}_{3}(g)$ takes, a priori, the three possible forms in Fig.\ref{Figure29}, where the values of degree $v_{i}$ discriminate between these possibilities. Recall that these three graphs originate from Newton graphs.
\end{proof}
The reasoning in the above Case $r=3$ does {\it not} imply that each of the graphs in Fig.\ref{Figure29} can be realized by a Newton flow. So we need:

\begin{lemma}
\label{L10.2}
If $r=2$ or $3$, then each pseudo Newton graph of the type $\check{\mathcal{G}}_{r}$ or $\hat{\mathcal{G}}_{r}$ can be represented as $\mathcal{G}_{r}(g)$, $g \in \underline{E}^{1}_{r}$.
\end{lemma}
\begin{proof}
$\underline{r=2}:$  Follows from 
Lemma \ref{L10.1}.\\
$\underline{r=3}:$ 
In Fig.\ref{IVFigure16} we consider the local phase portrait of $\overline{\mathcal{N}} (f)$ around the zero $v(=1+i)$ for $f$. Compare Fig.\ref{IVFigure11} and note that the zeros for $f$ are star nodes for $\overline{\mathcal{N}} (f)$. Since $\alpha +\gamma =\frac{1}{r} -\alpha$ and $\alpha +\beta +\gamma=\frac{1}{2}$ we have $\beta >   \frac{1}{2r}$ so that the angle $\beta+\beta$ spans an arc greater than $\frac{1}{r}(=\frac{1}{3})$.
Now the idea is:

To split off from the $3^{rd}$
order zero $v$ for $f$  a simple zero ($v_{1}$) ``Step 1'', and thereupon, to split up the remaining double zero ($v^{'}_{1}$) into two simple ones ($v_{2}$, $v_{3}$) ``Step 2'', in such a way that by an appropriate strategy, the resulting functions give rise to Newton flows with associated graphs, determining each of the three possible types in Fig.\ref{Figure29}.

\noindent
$\underline{{\rm Ad\; Step \:1}}:$ 
We perturb the original function $f$ into an elliptic function $g$ with one simple  ($v_{1}$) and one double ($v^{'}_{1}$ ) zero (close to each other), and one third order pole $w_{1}$
(thus close\footnote{\label{VNT11} Use property (\ref{vgl9}).} to the third order pole $w$ of $f$).
The original flow $\overline{\overline{\mathcal{N}}} (f)$ perturbs into a flow $\overline{\overline{\mathcal{N}}} (g)$ with $v_{1}$  and $v^{'}_{1}$ as attractors and $w_{1}$  as repellor. 
When $v_{1}$  tends to $v^{'}_{1}$, the perturbed function $g$ will tend to $f$, and thus the perturbed flow $\overline{\overline{\mathcal{N}}} (g)$ to $\overline{\overline{\mathcal{N}}} (f)$, cf. Subsubsection 1.1.2.
In particular, when the splitted zeros are sufficiently close to each other and the circle $C_{1}$ that encloses an open disk $D_{1}$ with center $v^{'}_{1}$, is chosen sufficiently small, $C_{1}$ is a global boundary (cf.\cite{JJT1}) for the perturbed flow $\overline{\overline{\mathcal{N}}}(g)$.
 It follows that, apart from the 
 equilibria $v_{1}$ and $v^{'}_{1}$  (both of Poincar\'e index 1) the flow $\overline{\overline{\mathcal{N}}} (g)$ exhibits on $D_{1}$ one other equilibrium 
 (with index $-1$): a simple saddle, say $c$ (cf. \cite{Guil}). 
 From this, it follows (cf. Subsubsection 1.1.1)
 that the phase portrait of $\overline{\overline{\mathcal{N}}} (g)$ around $v_{1}$  and $v^{'}_{1}$ is as sketched in Fig.\ref{Figure30}-(a), where the local basin of attraction for $v_{1}$ is shaded and intersects $C_{1}$ under an arc with length approximately $\frac{1}{3}$ . 
 
 \begin{figure}
\centering
\includegraphics[width=5.5in]{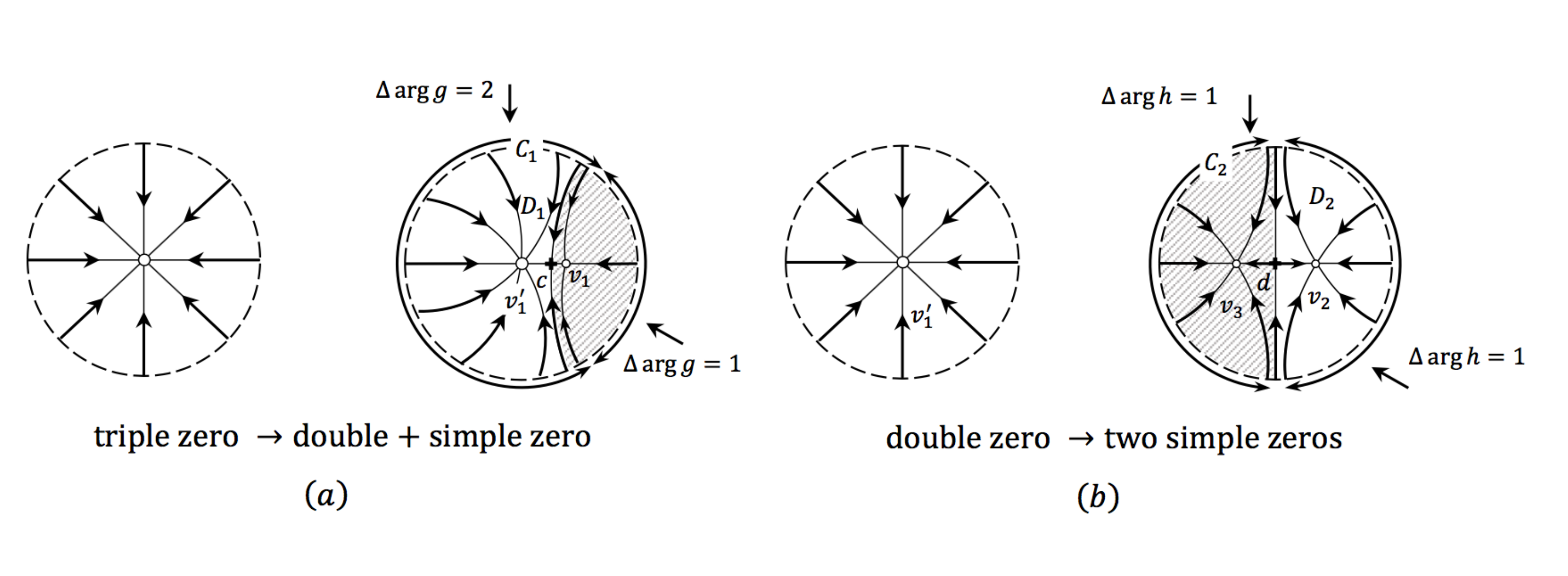}
\caption{Splitting zeros}
\label{Figure30}
\end{figure} 
 
 On the (compact!) complement  $T \backslash D_{1}$ this flow has one repellor ($w_{1}$)
 and two saddles.
 The repellor may be considered as hyperbolic (by the suitably chosen damping factor, compare the proof of Lemma \ref{L9.5}), whereas the saddles are distinct and thus simple (because $\overline{\overline{\mathcal{N}}} (f)$ has two simple saddles, say $\sigma_{1}$, $\sigma_{2}$, depending continuously on $v_{1}$ and $v^{'}_{1}$).
 Hence, the restriction of $\overline{\overline{\mathcal{N}}} (g)$ to $T \backslash D_{1}$ is $\varepsilon$-structurally stable (cf. \cite{Peix1}). So, we may conclude that, if $v_{1}$ (chosen sufficiently close to $v^{'}_{1}$) turns around $v^{'}_{1}$,  the phase portraits outside $D_{1}$ of the perturbed flows undergo a change that is negligible in the sense of the $C^{1}$-topology. Therefore, we denote the equilibria of $\overline{\overline{\mathcal{N}}} (g)$ on $T \backslash D_{1}$ by $w_{1}$, $\sigma_{1}$, $\sigma_{2}$ (i.e.,  without reference to $v_{1}$).
We move $v_{1}$ around a small circle, centered at  $v^{'}_{1}$ and focus on two positions (I, II) of $v_{1}$, specified by the position of $v_{1}$ w.r.t. the symmetry axis $\ell$.
 See Fig.\ref{IVFigure16} in comparison with Fig.\ref{IVFigure17}, where we sketched some trajectories of the phase portraits of $\overline{\mathcal{N}} (g)$ on $D_{1}$. \\
 
 \begin{figure}[h!]
\begin{center}
\includegraphics[scale=0.4]{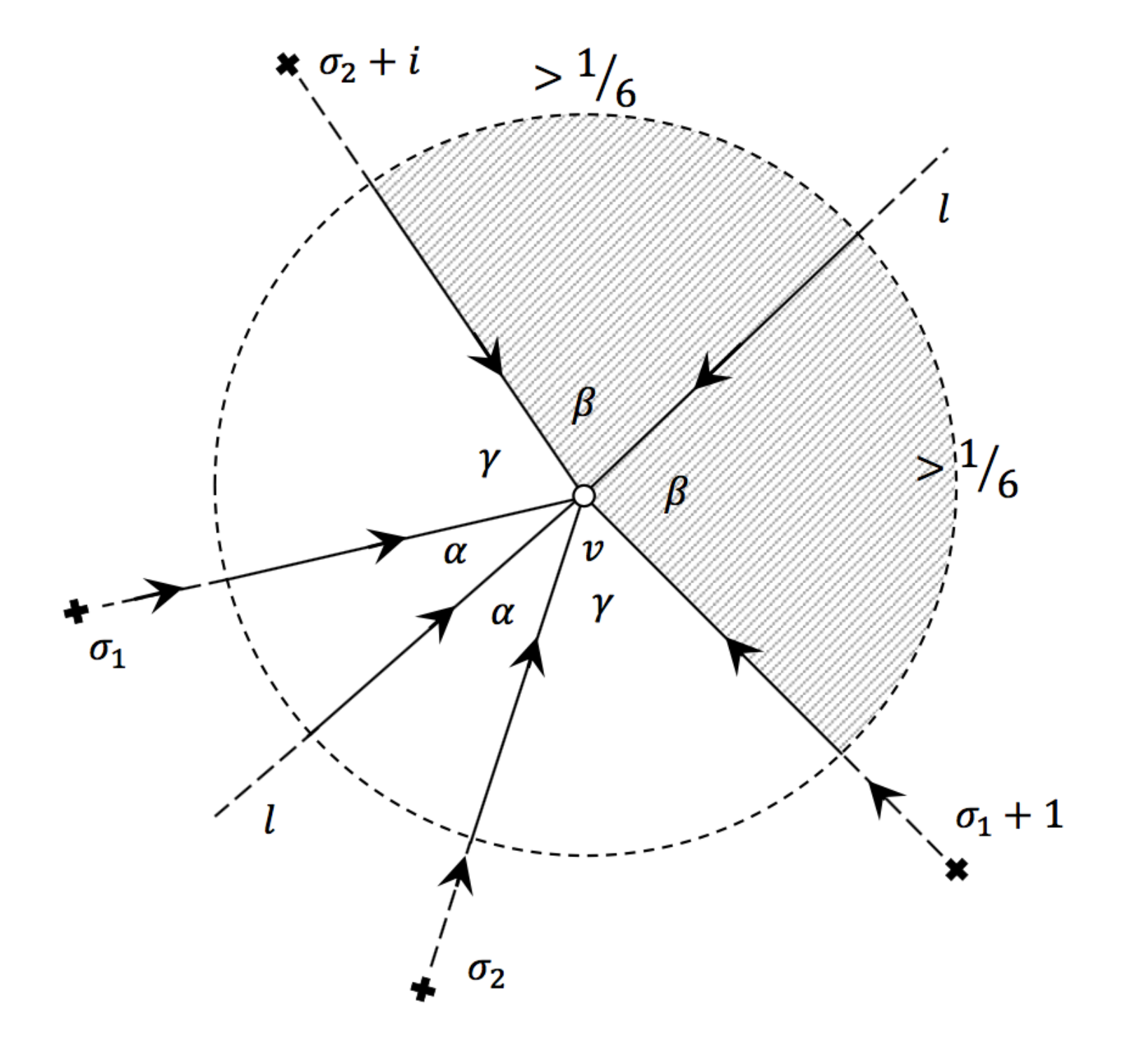}
\caption{\label{IVFigure16} Local phase portrait of $\overline{\overline{\mathcal{N}}} (f)$ around the zero $v(=1+i)$ for $f$; $r=2$.}
\end{center}
\end{figure}

\noindent
$\underline{{\rm Ad\; Step \:2}}:$
We proceed as in Step 1. Splitting $v^{'}_{1}$ into $v_{2}$  and $v_{3}$ (sufficiently close to each other) yields a perturbed elliptic function $h$, and thus a perturbed flow $\overline{\overline{\mathcal{N}}} (h)$. 
Consider a circle $C_{2}$, centered at the mid-point of $v_{2}$  and $v_{3}$, that encloses an open disk $D_{2}$ containing these points. 
If we choose $C_{2}$ sufficiently small, it is a global boundary of $\overline{\overline{\mathcal{N}}} (h)$. 
Reasoning as in Step 1, we find out that $\overline{\overline{\mathcal{N}}} (h)$ has on $D_{2}$  two simple attractors ($v_{2}$ , $v_{3}$ ) and one simple saddle: $d$ (close to the mid point of $v_{2}$  and $v_{3}$; compare Fig.\ref{Figure30}-(b)), where the local basin of attraction for $v_{3}$ is shaded and intersects $C_{2}$ under an arc with length approximately $\frac{1}{2}$. Moreover, as for $\overline{\overline{\mathcal{N}}} (g)$ in Step 1,  the flow $\overline{\overline{\mathcal{N}}} (h)$ is $\varepsilon$-structurally stable outside $D_{2}$. So, we may conclude that, if $v_{2}$  and $v_{3}$  turn (in diametrical position) around their mid-point, the phase portraits outside $D_{2}$  of the perturbed flows undergo a change that is negligible in the sense of $C^{1}$-topology. Therefore, we denote the equilibria of $\overline{\overline{\mathcal{N}}} (h)$ on $T \backslash D_{2}$ by $v_{1}$, $w_{1}$, $c$,  $\sigma_{1}$, and $\sigma_{2}$ (i.e., without reference to $v_{2}$  and $v_{3}$).

\begin{figure}[h!]
\begin{center}
\includegraphics[scale=0.4]{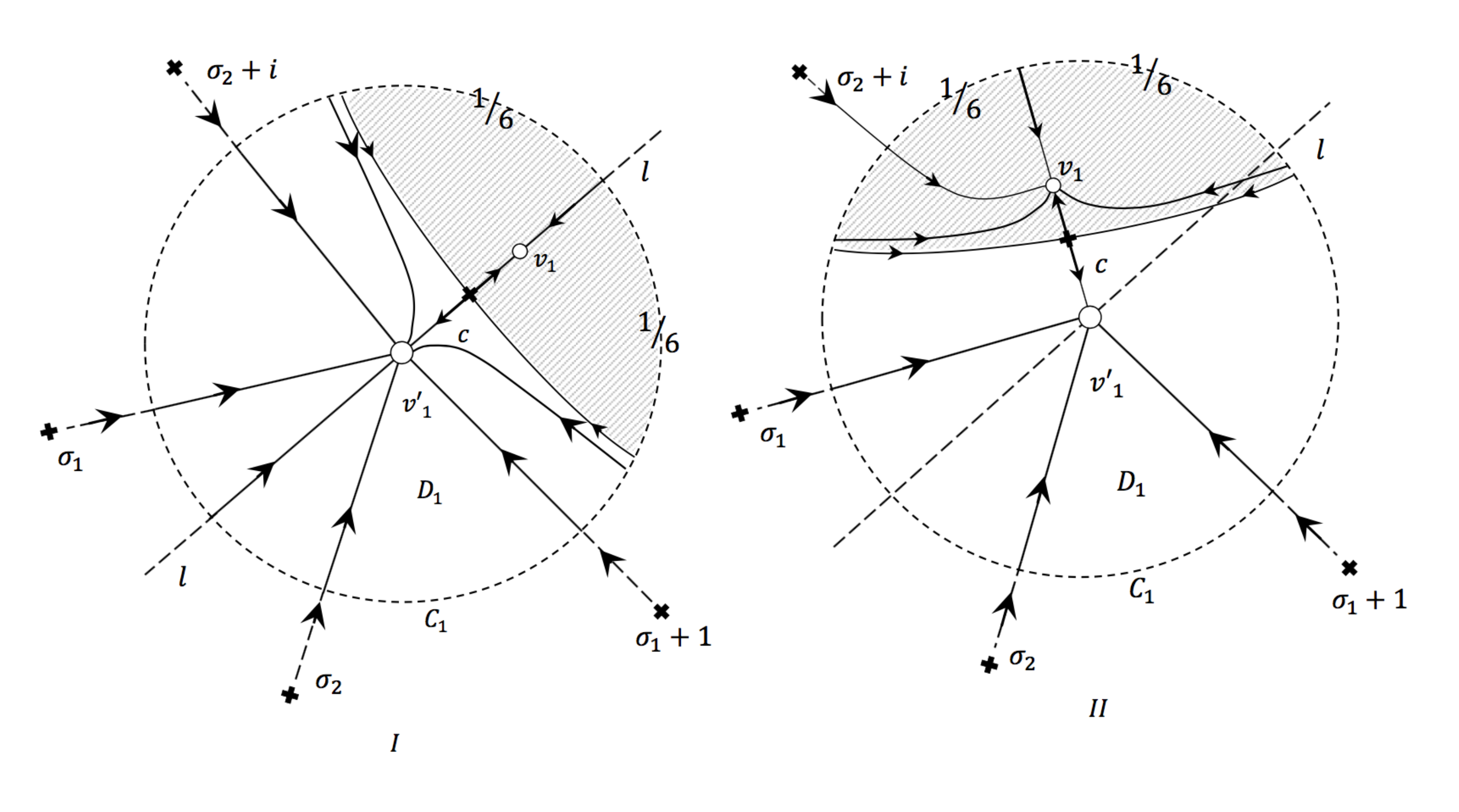}
\caption{\label{IVFigure17} Phase portraits for $\overline{\overline{\mathcal{N}}} (g)$) on  $D_{1}$.}
\end{center}
\end{figure}

\begin{figure}[h!]
\begin{center}
\includegraphics[scale=0.3]{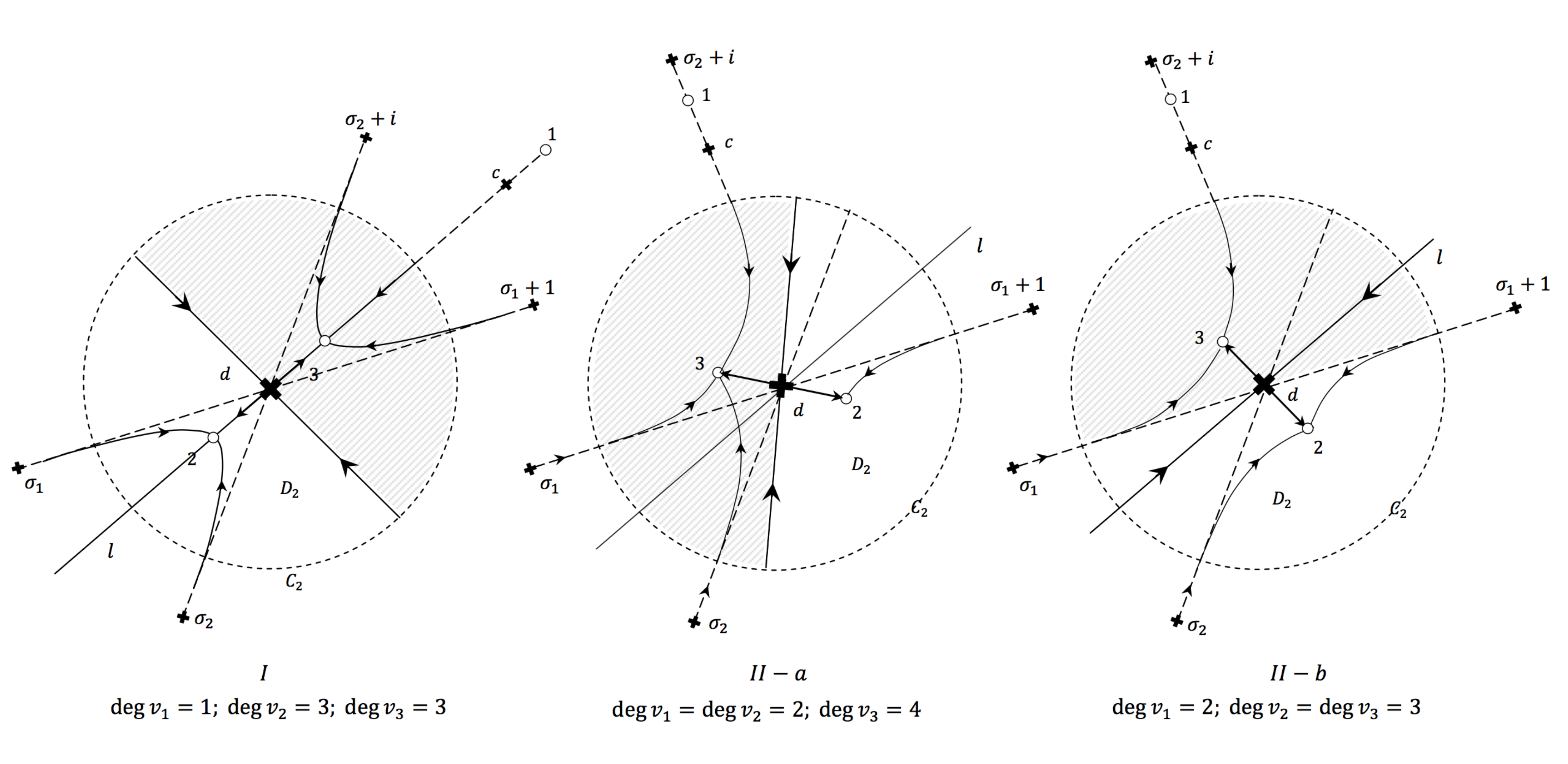}
\caption{\label{IVFigure18} The three phase portraits for $\overline{\overline{\mathcal{N}}} (h)$) on  $D_{2}$.}
\end{center}
\end{figure}

Finally, for $v_{1}$ in the position of Fig.\ref{IVFigure17}-(I)  we choose the pair ($v_{2}$,$v_{3}$) as in Fig.\ref{IVFigure18}-I; and for $v_{1}$ in the position of Fig.\ref{IVFigure17}-(II), we distinguish between two possibilities: Fig.\ref{IVFigure18}-IIa or Fig.\ref{IVFigure18}-IIb. 
Note that, with these choices of $v_{1}$, $v_{2}$, $v_{3}$ each of the obtained functions has three simple zeros and one triple pole.
Moreover, the four saddles are simple and not connected, whereas the three zeros are simple as well.
So the graph of the associated Newton flow is well defined and has only one face, four edges and three vertices.  Recall that the various values of 
degree $v_{i}$
discriminate between the three possibilities for the graphs of type $\check{\mathcal{G}}_{3}$). Now inspection of Fig.\ref{IVFigure18} yields the assertion.
\end{proof}

Up till now, we paid attention to pseudo Newton graphs with only one face (i.e., of type $\check{\mathcal{G}}$ or $\hat{\mathcal{G}}$.). 
If $r=2$, these are the only possibilities. \\
If $r=3$, there are also pseudo Newton graphs (denoted by $\underline{\mathcal{G}}$) with two faces
and angles summing up to 1 or 2. 
When the boundaries of any pair of the original $\mathcal{G}_{3}$-faces have a subwalk in common, these walks have length 1 or 2. (Use the A-property and compare Fig.\ref{Figure9.11}). 
So, when two $\mathcal{G}_{3}$-faces are merged, the resulting $\underline{\mathcal{G}}$-face admits {\it either} only vertices of degree $\geqslant 2$ {\it or} one vertex of degree 1\footnote{\label{FTN11}The E-property holds not always for $\underline{\mathcal{G}}$: Only if there is a $\underline{\mathcal{G}}$-vertex of degree 1, the dual $\underline{\mathcal{G}}^{*}$ admits a {\it contractible} loop (corresponding with the edge adjacent to this vertex); all other $\underline{\mathcal{G}}^{*}$-loops- if there are any- are {\it non-contractible}.}. From now on, we focus on the Newton graphs as exposured in Fig.\ref{Figure9.11} (i),(iv) [since all other Newton graphs (in this figure) can be dealt with in the same way, there is no loss of generality]. Then the two $\underline{\mathcal{G}}$-faces under consideration are $\underline{F}_{7}$ ($:=F_{4,6}$) and $\underline{F}_{8}:=F_{5}$; see Fig.\ref{Figure54} 
(in comparison with Fig.\ref{Figure9.11} (i),(iv)).

\begin{figure}[h!]
\centering
\includegraphics[scale=0.5]{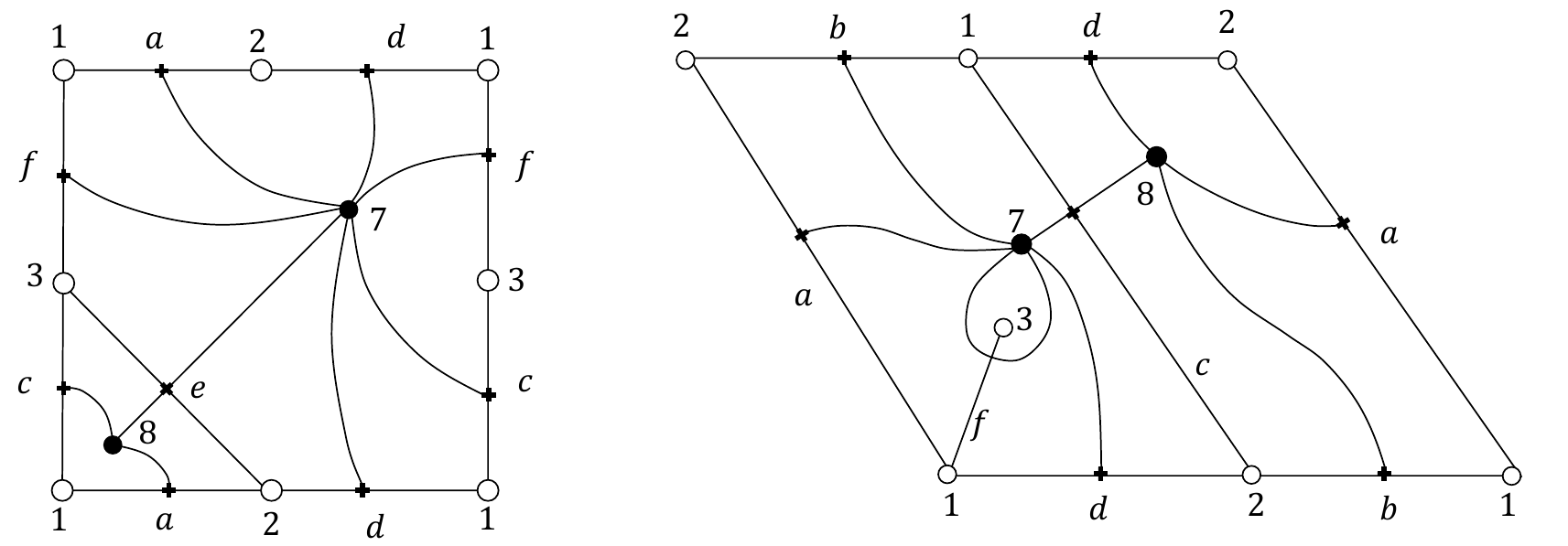}
\caption{The graph $\underline{\mathcal{G}} \wedge \underline{\mathcal{G}}^{*}$.}
\label{Figure54}
\end{figure}

We consider the common refinement $\underline{\mathcal{G}} \wedge \underline{\mathcal{G}}^{*}$ of $\underline{\mathcal{G}}$ and its dual $\underline{\mathcal{G}}^{*}$. Following Peixoto \cite{Peix1}, \cite{Peix2} we claim that $\underline{\mathcal{G}} \wedge \underline{\mathcal{G}}^{*}$ determines a $C^{1}$-structurally stable toroidal flow $X(\underline{\mathcal{G}})$ with canonical regions as depicted\footnote{\label{FTN13}In the terminology used in \cite{Peix1}, the canonical region in the r.h.s. of Fig.\ref{Figure55} is of Type 3, whereas the other two regions are of Type 1.} in Fig.\ref{Figure55}. As equilibria for $X(\underline{\mathcal{G}})$ we have: three stable and two unstable proper nodes (corresponding to the $\underline{\mathcal{G}}$-resp. 
$\underline{\mathcal{G}}^{*}$-vertices) and five orthogonal saddles (corresponding to the pairs $(e, e^{*})$ of $\underline{\mathcal{G}}$- and $\underline{\mathcal{G}}^{*}$-edges).

\begin{figure}[h!]
\centering
\includegraphics[scale=0.3]{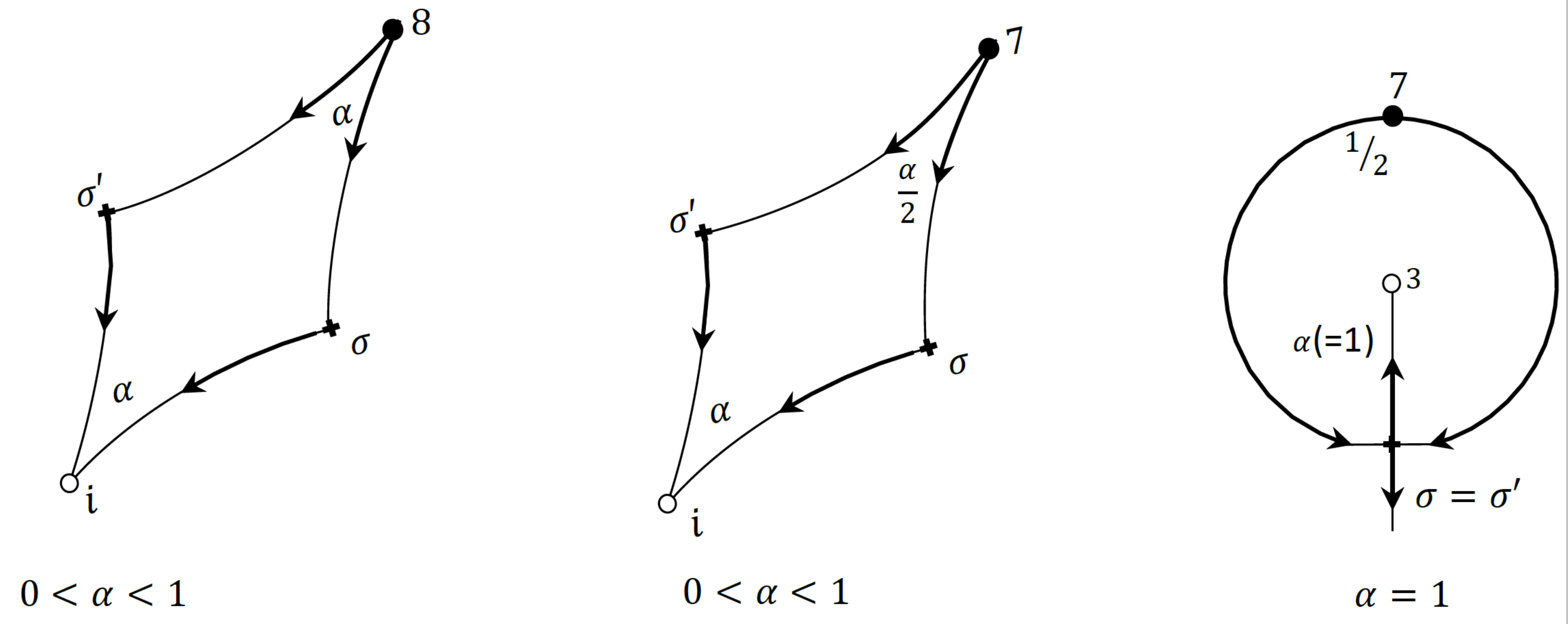}
\caption{The various appearances of the canonical regions of $\mathcal{X}(\underline{\mathcal{G}})$.}
\label{Figure55}
\end{figure} 

 Argueing basically as in the proof of Theorem 4.1 from our paper \cite{HT2}
, it can be shown that  $X(\underline{\mathcal{G}})$) is equivalent with an elliptic Newton flow generated by a function on three simple zeros, one double and one simple pole and five simple critical points; compare\footnote{\label{FTN14} The picture in the r.h.s. of Fig.\ref{Figure55} corresponds to a  canonical region in the $\underline{\mathcal{G}}$-face
 with angles summing up to 1, determining the simple pole for $X(\underline{\mathcal{G}})$, compare also Fig.\ref{Figure2N}. The other two pictures correspond to the $\underline{\mathcal{G}}$-face with angles summing up to 2, determining the double pole for $X(\underline{\mathcal{G}})$.} Fig.\ref{IVFigure10}. As an elliptc Newton flow, $X(\underline{\mathcal{G}})$ is not ($\tau_{0}$-)structurally stable, since $\underline{\mathcal{G}}$ is not Newtonian, compare Subsubsection 1.2.2. However, by the aid of a suitably chosen damping factor, compare Lemma \ref{L9.5} (preambule), and within the class of all elliptic Newton flows generated by functions on three simple zeros, on one double and one simple pole and on five simple critical points, the flow $X(\underline{\mathcal{G}})$ is structurally stable w.r.t. the relative topology $\tau_{0}$.\\
 
 Altogether, we find:
 
 \begin{theorem}
\label{T4.3}
Any pseudo Newton graph of order $r, r=2, 3$ represents an elliptic Newton flow. In particular,  the $3^{\text{rd}}$
order nuclear Newton flow ``creates'' - by splitting up zeros/ poles $($``bifurcation''$ )$-  all, 
up to duality and topological equivalency, structurally stable elliptic Newton flows of order 3. 
\end{theorem}

\end{document}